\documentclass{amsart}

\newcommand{\R}{{\mathbb R}}
\newcommand{\N}{{\mathbb N}}
\newcommand{\Z}{{\mathbb Z}}
\newcommand{\on}{\operatorname}
\newcommand{\eps}{\varepsilon}
\newtheorem{definition}{Definition}[section]
\newtheorem{bem}[definition]{Remark}

\newtheorem{theorem}[definition]{Theorem}
\newtheorem{proposition}[definition]{Proposition}
\newtheorem{lemma}[definition]{Lemma}
\newtheorem{corollary}[definition]{Corollary}
\makeatletter

\begin{document}
\title{Geometric flows with rough initial data}
\author{Herbert Koch}
\address[H.~Koch]{Mathematisches Institut\\Universit{\"a}t Bonn\\Endenicher Allee 60\\53115 Bonn\\Germany}
\email{koch@math.uni-bonn.de}
\author{Tobias Lamm}
\address[T.~Lamm]{Department of Mathematics\\University of British Columbia\\1984 Mathematics Road\\Vancouver, BC V6T 1Z2\\Canada}
\email{tlamm@math.ubc.ca}
\thanks{The second author is partially supported by a PIMS Postdoctoral Fellowship.}
\date{\today}
\begin{abstract}
We show the existence of a global unique and analytic solution for the
  mean curvature flow, the surface diffusion flow and the Willmore flow of entire graphs for Lipschitz initial data with small Lipschitz norm. We also show the existence of a global unique and analytic solution to the Ricci-DeTurck flow on euclidean space for bounded
  initial metrics which are close to the euclidean metric in
  $L^\infty$ and to the harmonic map flow for initial maps whose image
  is contained in a small geodesic ball.
\end{abstract}
\maketitle
\section{Introduction}

In this paper we prove the existence of solutions of geometric flows with non-smooth initial data. More precisely, we consider the graphical Willmore, surface diffusion and mean curvature flow, the Ricci-DeTurck flow on $\R^n$ and the harmonic map flow for maps from $\R^n$ into a compact target manifold.

The initial data we are interested in are Lipschitz functions for the mean curvature, surface diffusion and Willmore flow, and $L^\infty$ metrics (respectively maps) for the Ricci-DeTurck and harmonic map flow. Here and in the rest of the paper we say that a function $f$ is Lipschitz if it belongs to the homogeneous Lipschitz space $C^{0,1}(\R^n)$ with norm $||f||_{C^{0,1}(\R^n)}=||\nabla f||_{L^\infty(\R^n)}$. We construct the solutions of the flows via a fixed point argument and therefore we require the initial data to be small in the corresponding spaces.
 
Crucial in our construction are scale invariant norms based on space-time cylinders, similar to the Carleson weight characterization of $BMO$ (see \cite{stein93, koch01}). This point of view has been introduced by the first author and Tataru \cite{koch01} in the context of the Navier-Stokes equations. Here we approach quasilinear equations and we obtain new and possibly optimal results in terms of the regularity of the initial data and the regularity of the solution. Moreover our method to construct the solutions allows a uniform and efficient treatment of the five geometric evolution equations.
 
In the above mentioned paper \cite{koch01} a fixed point argument was used in order to show the existence of a unique global solution of the Navier-Stokes equations for any initial data
which is divergence free and small in $BMO^{-1}$ (the space of distributions which are the divergence of vector fields with $BMO$ components). By localizing their construction the authors were also able to show the existence of a unique local solution of the Navier-Stokes equations for any initial data which is divergence free and in $VMO^{-1}$.

In the case of the harmonic map flow we show how a similar local construction can be used to obtain the existence of a local unique solution for initial maps which are small $L^\infty$-perturbations of uniformly continuous maps. 

Using an idea introduced by Angenent \cite{angenent90}, \cite{angenent90a} we obtain in all cases analyticity of the solution as a byproduct of the fixed point argument. 

It is likely that related local constructions can be used to obtain unique local and analytic solutions of the mean curvature, surface diffusion and Willmore flow of entire graphs for $C^1$ initial surfaces and even for small Lipschitz perturbations of such surfaces. This remark may be of interest for numerical approximations by triangulated surfaces. 

In the following we give a brief outline of the paper. 

In section $2$ we recall some basic properties of the heat kernel and the biharmonic heat kernel and we study solutions of the homogeneous linear equations with rough initial data.

In section $3$ we show the existence of a global unique and analytic solution of the Willmore and surface diffusion flow of entire graphs for Lipschitz initial data with small Lipschitz norm. Moreover we show the existence of global unique and analytic self-similar solutions for self-similar Lipschitz initial data having small Lipschitz norm.

A global unique and analytic solution to the Ricci-DeTurck flow on $\R^n$ for $L^\infty$-initial metrics which are $L^\infty$ close to the euclidean metric is constructed in section $4$. This yields a slight improvement of a recent existence result of Schn\"urer, Schulze \& Simon \cite{schnuerer08}.

In section $5$ we show the existence of a global unique and analytic solution of the mean curvature flow of entire graphs for Lipschitz initial data with small Lipschitz norm. We emphasize that this construction includes the case of higher codimensions.  

In section $6$ we construct a local unique solution of the harmonic map for every initial data which is a $L^\infty$-perturbation of a uniformly continous map. As a Corollary we get the existence of a global solution for the harmonic map flow for every initial map whose image is contained in a small geodesic ball.

Finally, in the appendix, we use the method of the stationary phase to derive some standard estimates for the biharmonic heat kernel.

\section{Preliminaries}
\setcounter{equation}{0} 
In this section we recall some estimates for the heat kernel and the biharmonic heat kernel and we prove estimates for solutions of the corresponding homogeneous initial value problems with rough initial data.
\subsection{Heat kernel}
The heat kernel $\Phi(x,t)=(4\pi t)^{-\frac{n}{2}}e^{-{\frac{|x|^2}{4t}}}$ is the fundamental solution of the heat equation
\begin{align*}
(\partial_t-\Delta)\Phi(x,t)=0\ \ \ \text{on}\ \ \R^n\times (0,\infty).
\end{align*}
We have the following estimates for the heat kernel and its derivatives. 
\begin{lemma}\label{kernela}
We have for every $k,l\in \N_0$ and $t>0$, $x\in \R^n$
\begin{align}
 |\partial_t^l \nabla^k \Phi(x,t)| &\le c (t^{\frac{1}{2}}+|x|)^{-n-k-2l}\ \ \ \text{and} \label{kernel2b}\\
\| \partial_t^l \nabla^k \Phi(\cdot,t) \|_{L^1(\R^n)} &\le c t^{-l-\frac{k}{2}} . \label{kernel2c}
\end{align}
Moreover, for any $(x,t) \in \R^n \times (0,1)\backslash \big(B_1(0)\times (0,\frac{1}{4})\big)$, there exist constants $c,c_1>0$ such that
\begin{align}
|\Phi(x,t)|+|\nabla \Phi|(x,t)+|\nabla^2 \Phi|(x,t)\le ce^{-c_1|x|}. \label{phi}
\end{align}
\end{lemma}

We note that solutions of the heat equation which grow slower than $e^{|x|^2}$ at infinity are unique (see for example \cite{evans}). In the following, whenever we speak of a solution of the heat equation, we mean a solution satisfying this growth condition. 

As a consequence of the estimates for the heat kernel we get the following result for solutions of the homogeneous heat equation. 
\begin{lemma}\label{homog2nd}
Let $u_0\in L^\infty(\R^n)$ and let $u:\R^n\times \R^+_0\rightarrow \R$ be a solution of the homogeneous linear equation 
\[ u_t - \Delta u = 0 , \qquad u(\cdot,0) = u_0. \]
Then we have 
\begin{align} 
|| u ||_{L^\infty(\R^n\times \R^+)} &+\sup_{t>0}t^{\frac{1}{2}}||\nabla u(t)||_{L^\infty(\R^n)}+  \sup_{x\in \R^n} \sup_{R>0} \Big(R^{-\frac{n}{2}} || \nabla u||_{L^2(B_R(x)\times (0,R^2))}\nonumber \\
&+R^{\frac{2}{n+4}}||\nabla u||_{L^{n+4}(B_R(x)\times (\frac{R^2}{2},R^2))} \Big)\nonumber \\
&\le c || u_0 ||_{L^\infty(\R^n)}.\label{homog2nda} 
\end{align} 
\end{lemma} 
\begin{proof}
The estimate \eqref{homog2nda} is invariant under translations and the scaling ($\lambda>0$) $u_\lambda(x,t)=u(\lambda x, \lambda^2 t)$ and therefore it suffices to show
\begin{align*} 
|u(0,1)| +|\nabla u(0,1)|&+  ||\nabla u||_{L^2(B_1(0)\times (0,1))}+||\nabla u||_{L^{n+4}(B_1(0)\times (\frac12,1))} \\
&\le c || u_0 ||_{L^\infty(\R^n)}.
\end{align*} 
Now \eqref{kernel2c} implies that for $i\in \{0,1\}$ we have 
\begin{align*}
\sup_{x\in B_1(0)} \sup_ {\frac{1}{2}\le t\le 1}|\nabla^i u(x,t)|&\le \sup_{x\in B_1(0)} \sup_{\frac{1}{2}\le t\le 1}|\int_{\R^n}\nabla^{i} \Phi(y,t)  u_0(x-y) dy|\\
&\le c||u_0||_{L^\infty(\R^n)}.
\end{align*} 
In order to estimate the third term we let $\eta\in C^\infty_c(B_2(0))$, $0\le \eta \le 1$, $\eta \equiv 1$ in $B_1(0)$ with $||\nabla \eta||_{L^\infty(\R^n)}\le c$ be a standard cut-off function. Multiplying the homogeneous heat equation with $\eta^2 u$ and integrating by parts we get with the help of Young's inequality and the pointwise estimate for $u$
\begin{align*}
\partial_t  \int_{\R^n}\eta^2 |u|^2+ \int_{\R^n} \eta^2 |\nabla u|^2 \le c\int_{B_2(0)} |\nabla \eta|^2 |u|^2 \le c|| u_0 ||^2_{L^\infty(\R^n)}. 
\end{align*}
Integrating this estimate from $0$ to $1$ and using the properties of $\eta$ yields the desired result.
\end{proof}

\begin{bem} 
The choice of the $L^2$-spacetime norm of the gradient is motivated by the Carleson measure characterization of $BMO(\R^n)$ (see \cite{stein93, koch01}). Namely, for a solution $u$ of $\partial_t u-\Delta u=0$ on $\R^n\times (0,\infty)$ with $u(\cdot,0)=u_0$, we have 
\begin{align*}
||u_0||_{BMO(\R^n)}=\sup_{x\in \R^n} \sup_{R>0}R^{-\frac{n}{2}} || \nabla u||_{L^2(B_R(x)\times (0,R^2))}
\end{align*}
in the sense that the right hand side defines an equivalent norm for $BMO(\R^n)$. 
\end{bem}

\subsection{Biharmonic heat kernel}
The biharmonic heat kernel $b(x,t)$ is the fundamental solution of 
\begin{align*}
(\partial_t+\Delta^2)b(x,t)=0\ \ \ \text{on}\ \ \R^n\times (0,\infty)
\end{align*}
and it is given by
\begin{align*}
 b(x,t)=&  \mathcal{F}^{-1}(e^{-|k|^4t})\\
=& t^{-\frac{n}{4}}g(\eta),
\end{align*}
where $\eta=xt^{-\frac{1}{4}}$ and 
\[
 g(\eta)= (2\pi)^{-\frac{n}{2}} \int_{\R^n}  e^{i\eta k -|k|^4} dk. 
\]
We have the estimate
\begin{align*}
 |g(\eta)|\le K (1+|\eta|)^{-\frac{n}{3}} e^{- \alpha |\eta|^{\frac{4}{3}}}
\end{align*}
with $\alpha= 2^{\frac13} \frac3{16} $.  
Additionally we have for every $m \in \N$ that
\[
 |\frac{d^m g}{d \eta^m}(\eta)| \le K_m  (1+|\eta|)^{-\frac{n-m}3}    e^{- \alpha |\eta|^{\frac{4}{3}}}. 
\]
Standard proofs of these estimates are provided in appendix A.
In the following Lemma we rephrase the above estimates on $b$ and its derivatives in such a way that we can directly apply them later on.
\begin{lemma}\label{kernel}
We have for every $t>0$ and $x\in \R^n$ that
\begin{align}
|b(x,t)|\le  c t^{-\frac{n}{4}} exp\Big(-\alpha \frac{|x|^{\frac{4}{3}}}{t^{\frac{1}{3}}}\Big).\label{kernel1}
\end{align}
Moreover we have for every $k,l\in \N_0$ and $t>0$, $x\in \R^n$
\begin{align}
 |\partial_t^l \nabla^k b(x,t)| &\le c (t^{\frac{1}{4}}+|x|)^{-n-k-4l} \ \ \ \text{and} \label{kernel2}\\
\| \partial_t^l \nabla^k b(\cdot,t) \|_{L^1(\R^n)} &\le c t^{-l-\frac{k}{4}}. \label{kernel2a}
\end{align}
Finally, for all $(x,t) \in \R^n\times (0,1) \backslash \big( B_1(0)\times (0,\frac{1}{4}) \big)$ and all $0\le j \le 4$ there exist constants $c,c_1>0$ such that
\begin{align} 
|\nabla^j b(x,t)|\le c e^{- c_1|x|}.\label{decay} 
\end{align}
\end{lemma}

Solutions of the biharmonic heat equation which grow slower than $e^{|x|^{\frac43}}$ at infinity are unique, and therefore, whenever we speak of a solution of the biharmonic heat equation we mean the one which satisfies this growth condition.
 
We also need the following estimate for solutions of the homogeneous problem.
\begin{lemma}\label{homog}
Let $u:\R^n\times \R^+_0\rightarrow \R$ be a solution of the homogeneous linear equation 
\[ u_t + \Delta^2 u = 0 , \qquad u(\cdot,0) = u_0\in C^{0,1}(\R^n). \]
Then we have 
\begin{align} 
|| \nabla u ||_{L^\infty(\R^n\times \R^+)} &+\sup_{t>0}t^{\frac{1}{4}}||\nabla^2 u(t)||_{L^\infty(\R^n)}\nonumber \\
&+  \sup_{x\in \R^n} \sup_{R>0}  R^{\frac{2}{n+6}} || \nabla^2 u||_{L^{n+6}(B_R(x)\times (\frac{R^4}{2},R^4))} \nonumber \\
\le& c || u_0 ||_{C^{0,1}(\R^n)}.\label{homog1} 
\end{align} 
\end{lemma} 
\begin{proof}
Since the estimate \eqref{homog1} is invariant under translations and the scaling $u_\lambda(x,t)=\frac{1}{\lambda} u(\lambda x,\lambda^4 t)$ ($\lambda>0$), it suffices to show that
\begin{align*} 
| \nabla u (0,1)| +|\nabla^2 u(0,1)|+ || \nabla^2 u||_{L^{n+6}(B_1(0)\times (\frac{1}{2},1))} \le c || u_0 ||_{C^{0,1}(\R^n)}.
\end{align*} 
Using \eqref{kernel1} and \eqref{kernel2a} we get for $i=1,2$
\begin{align*}
\sup_{x\in B_1(0)} \sup_ {\frac{1}{2}\le t\le 1}|\nabla^i u(x,t)|&\le \sup_{x\in B_1(0)} \sup_{\frac{1}{2}\le t\le 1}|\int_{\R^n}\nabla^{i-1} b(y,t) \nabla u_0(x-y) dy|\\
&\le c||u_0||_{C^{0,1}(\R^n)}.
\end{align*} 
This finishes the proof of the Lemma.
\end{proof}

\section{Willmore and surface diffusion flow}
\setcounter{equation}{0}
For a closed two-dimensional surface $\Sigma$ and an immersion $f:\Sigma \rightarrow \R^3$ the Willmore functional is defined by 
\begin{align}
W(f)=\frac{1}{4} \int_\Sigma H^2 d\mu_g,\label{energy}
\end{align}
where $g$ is the induced metric, $H=\kappa_1+\kappa_2$ is the mean curvature of $\Sigma$ and $d\mu_g$ is the area element. Critical points of $W$ are called Willmore surfaces and they are solutions of the Euler-Lagrange equation
\begin{align}
\Delta_g H + \frac{1}{2}H^3-2HK=0, \label{EL}
\end{align}
where $\Delta_g$ is the Laplace-Beltrami operator of the induced metric and $K=\kappa_1 \kappa_2$ is the Gauss curvature of $\Sigma$. The Willmore flow is the $L^2$-gradient flow of $W$ and is therefore given by the following fourth order quasilinear parabolic equation
\begin{align}
f_t^\perp &=-\Delta_g H - \frac{1}{2}H^3+2HK \ \ \ \text{on}  \ \ \Sigma \times [0,T), \nonumber \\
f(\cdot,0)&=f_0 ,\label{flow}
\end{align}
where $f_0:\Sigma \rightarrow \R^3$ is some given immersion and $f_t^\perp$ denotes the normal part of $f_t$. In the case that $\Sigma$ is a sphere Kuwert \& Sch\"atzle \cite{K-S1}-\cite{K-S2} showed that if $W(f_0)\le 8\pi$, then the Willmore flow exists for all times and subconverges to a smooth Willmore sphere. On the other hand Mayer \& Simonett \cite{mayer02} gave a numerical example for a singularity formation of the Willmore flow for an initial immersion of a sphere $f_0$ with $W(f_0)<8\pi +\varepsilon$, where $\varepsilon>0$ is arbitrary (for an analytic proof of this result see \cite{blatt}). Moreover, in a recent paper, Chill, Fasangova \& Sch\"atzle \cite{chill08} showed that if $f_0$ is $W^{2,2}\cap C^1$ close to a $C^2$ local minimizer of $W$ (i.e. a minimizer among all closed immersions which are $C^2$ close to each other), then the Willmore flow with initial data $f_0$ exists for all times and converges (after reparametrization) to a $C^2$ local minimizer of $W$.

In this section we are interested in the Willmore flow for graphs on $\R^2$ (so called entire graphs). Hence we assume that there exists a function $u:\R^2 \rightarrow \R$ such that $\Sigma=\text{graph}(u)=\{(x,u(x))| x\in \R^2\}$. Standard calculations then yield
\begin{align*}
f_t^\perp&= \frac{u_t}{v},\ \ \ H= \on{div} (\frac{\nabla u}{v}), \ \ \ K= \frac{\det \nabla^2 u}{v^4} \ \ \ \text{and} \\
\Delta_g H &=\frac{1}{v}\on{div} \Big((vI -\frac{\nabla u \otimes \nabla u}{v}) \nabla H \Big), 
\end{align*}
where $v=\sqrt{1+|\nabla u|^2}$. From the calculations in \cite{deckelnick06} we get
\begin{align*}
  \Delta_g H + \frac{1}{2}H^3-2HK= \on{div} \Big( \frac{1}{v}
  \big((I-\frac{\nabla u \otimes \nabla u}{v^2}) \nabla
  (vH)-\frac{1}{2}H^2 \nabla u\big) \Big)
\end{align*}
and therefore the Willmore flow equation \eqref{flow} can be rewritten as
\begin{align}
  u_t+v\on{div} \Big( \frac{1}{v} \big((I-\frac{\nabla u \otimes
    \nabla u}{v^2}) \nabla (vH)-\frac{1}{2}H^2 \nabla u\big) \Big)=0 \
  \ \ \text{on} \ \ \R^2 \times [0,T) \label{wflow}
\end{align}
with initial condition $u(\cdot, 0)=u_0$, where $u_0:\R^2 \rightarrow
\R$ is some function. The following observation concerning the scaling behavior of a solution of the Willmore flow turns out to be very important: If $u(x,t)$ is a solution of \eqref{wflow} with initial condition $u(\cdot,0)=u_0$, then the
rescaled function
\begin{align}
 u_\lambda(x,t)= \frac{1}{\lambda}u(\lambda x, \lambda^4 t)\label{scaling}
\end{align}
is also a solution of \eqref{wflow} with initial condition $u_\lambda(\cdot,0)=u_{\lambda,0}=\frac{1}{\lambda}u_0(\lambda \cdot)$. 

Our aim in this section is to show the existence of a global unique and analytic solution of \eqref{wflow} under very weak regularity assumptions on the initial data $u_0$. Since the method we use to construct the solution does not depend on the dimension we consider in the following solutions of \eqref{wflow} on $\R^n$ ($n\in \N$). 

Before stating the main result of this section we need to define a suitable Banach space. For functions $u:\R^n\times (0,\infty)\rightarrow \R$ which are continuous and twice differentiable with respect to the space variable we define a norm by
\[
||u||_{X_\infty}:=\sup_{t>0} || \nabla u(t) ||_{L^\infty(\R^n)} +  \sup_{x\in \R^n} \sup_{R>0} 
R^{\frac{2}{n+6}}  || \nabla^2 u||_{L^{n+6}(B_R(x)\times (\frac{R^4}{2},R^4))}.
\]
Moreover we define the Banach space 
\begin{align*}
 X_\infty=& \{u |\ \ ||u||_{X_\infty}<\infty \}.
\end{align*}
The following Theorem is our main result for the Willmore flow of graphs.
\begin{theorem}\label{main}
  There exists $\varepsilon>0$, $C>0$ such that for every $u_0\in
  C^{0,1}(\R^n)$ satisfying $||u_0||_{C^{0,1}(\R^n)}<\varepsilon$
  there exists an analytic solution $u\in X_\infty$ of the Willmore flow \eqref{wflow} with $u(\cdot,0)=u_0$ which satisfies $||u||_{X_\infty}\le C||u_0||_{C^{0,1}(\R^n)}$. The solution is unique in the ball $B^{X_\infty}(0,C\eps)=\{u\in X_\infty |\ \ ||u||_{X_\infty}\le C\eps\}$. 

More precisely, there exist $R>0$, $c>0$ such that for every $k\in \N_0$ and multiindex $\alpha\in \N_0^n$ we have the estimate
\begin{align}
  \sup_{x\in \R^n}\sup_{ t>0} |(t^{\frac{1}{4}} \nabla)^\alpha (t\partial_t)^k \nabla u(x,t)|\le c||u_0||_{C^{0,1}(\R^n)}R^{|\alpha|+k}(|\alpha|+k)!. \label{ana}
\end{align}  
Moreover the solution $u$ depends analytically on $u_0$. 
\end{theorem}
We remark that in the above Theorem the initial value $u_0$ is allowed
to have infinite Willmore energy. Moreover weak solutions in $X_\infty$ are fixed points of the fixed point map $F_{u_0}$, defined below.

It is an interesting open problem if one can drop the
smallness assumption on the Lipschitz norm of $u_0$ in Theorem
\ref{main}.
\newline 
In the case of the mean curvature flow for entire
graphs, Ecker \& Huisken \cite{ecker89,ecker91} showed the
existence of a global solution for initial data which are locally
Lipschitz continuous. Since the equation \eqref{wflow} is of fourth
order it is not clear if one can expect a corresponding result
in this situation.

We would like to remark that the Willmore flow for graphs has
previously been studied from a numerical point of view by Deckelnick
\& Dziuk \cite{deckelnick06}.

In order to show the existence of a solution if the Willmore flow we need to rewrite the equation \eqref{wflow}. We start by introducing some notation. We use
the $\star$ notation to denote an arbitrary linear combination of
contractions of indices for derivatives of $u$. For example we have
$\nabla^2_{ij}u \nabla_i u \nabla_j u=\nabla^2 u \star \nabla u \star
\nabla u=|\nabla u|^2\Delta u$. Moreover we use the abstract notation
\begin{align*}
P_i(\nabla u)=\underbrace{\nabla u \star \ldots \star \nabla u}_{\text{$i$-times}}.
\end{align*}
With the help of this notation we are now able to rewrite equation \eqref{wflow} in a form which is more suitable for our purposes.
\begin{lemma}\label{rewr}
The Willmore flow equation \eqref{wflow} can be written as
\begin{align}
 u_t+\Delta^2 u= f_0[u]+\nabla_i f^i_1[u]+\nabla^2_{ij} f_2^{ij}[u]=:f[u], \label{rewr1}
\end{align}
where 
\begin{align*}
 f_0[u|&=  \nabla^2 u \star \nabla^2 u\star \nabla^2 u \star \sum_{k=1}^4 v^{-2k} P_{2k-2}(\nabla u) ,\\
f_1[u]&=\nabla^2 u \star \nabla^2 u \star \sum_{k=1}^4 v^{-2k} P_{2k-1}(\nabla u)\ \ \text{and}\\
f_2[u]&= \nabla^2 u\star \sum_{k=1}^2 v^{-2k} P_{2k}(\nabla u).
\end{align*}
\end{lemma}
\begin{proof}
We calculate term by term.
\begin{align*}
\frac{v}{2} \on{div} ( \frac{\nabla u}{v} H^2)=& \on{div} (\frac{\nabla u H^2}{2})-\frac{\nabla v \nabla u}{2v} H^2\\
=& \nabla_i f_1^i[u] +f_0[u]
\end{align*}
\begin{align*}
v \on{div} ( \frac{1}{v} \nabla (vH)) =& \Delta (vH) -\frac{\nabla v}{v} \nabla (vH)\\
=& \Delta^2 u -\Delta (\frac{\nabla v \nabla u}{v})-\frac{\nabla v}{v} \nabla \Delta u+\frac{\nabla v}{v} \nabla (\frac{\nabla v \nabla u}{v})\\
=& \Delta^2 u+ \nabla^2_{ij} f_2^{ij}[u]-I+II.
\end{align*}
Now we can rewrite $I$ as follows:
\begin{align*}
 I=& \nabla_j (\frac{\nabla_k u}{v^2}\nabla^2_{ik} u \nabla^2_{ij} u)-f_0[u]-\frac{\nabla_k u}{v^2} \nabla^3_{ijk} u \nabla^2_{ij} u\\
=& f_0[u]+\nabla_j f_1^j[u]-\frac{\nabla_k u}{2v^2} \nabla_k |\nabla^2 u|^2\\
=& f_0[u]+\nabla_j f_1^j[u]-\nabla_k (\frac{\nabla_k u}{2v^2}  |\nabla^2 u|^2)\\
=& f_0[u]+\nabla_j f_1^j[u].
\end{align*}
For $II$ we argue similarly to get
\begin{align*}
 II=&\frac{ \nabla_j u \nabla_k u \nabla_l u}{v^4}\nabla^2_{ik} u \nabla^3_{ijl} u +f_0[u]\\
=& \frac{ \nabla_j u \nabla_k u \nabla_l u}{2v^4}\nabla_j (\nabla^2_{ik} u \nabla^2_{il} u)+f_0[u]\\
=& f_0 [u]+\nabla_j f_1^j[u].
\end{align*}
Finally we have
\begin{align*}
 v\nabla_i(\frac{\nabla_i u \nabla_j u}{v^3} \nabla_j(vH))=& \nabla_i(\frac{\nabla_i u \nabla_j u}{v^2} \nabla_j(vH))-\frac{\nabla_i v \nabla_i u \nabla_j u}{v^3} \nabla_j(vH)\\
=& \nabla^2_{ij}(\frac{\nabla_i u \nabla_j u}{v} H)+\nabla_i f_1^i[u] +f_0[u],
\end{align*}
where we argued as above to rewrite the last term in the first line. Altogether this finishes the proof of the Lemma.
\end{proof}

Next we write equation \eqref{rewr1} in integral form
\begin{align}
 u(x,t)= Su_0(x,t)+Vf[u](x,t),\label{integral}
\end{align}
where 
\begin{align*}
 Su_0 (x,t)=&  \int_{\R^n} b(x-y,t)u_0(y)dy \ \ \ \text{and}\\
Vf[u](x,t)=& \int_0^t \int_{\R^n} b(x-y,t-s)f[u](y,s)dyds.
\end{align*}
The goal for the rest of this section is to construct a solution of the integral equation \eqref{integral} by using a fixed point argument.

Another very important fourth order geometric evolution equation is the surface diffusion flow, given by
\[
f_t^\perp =-\Delta_g H.
\]
For results on this flow see for example \cite{deckelnick05}.

Restricting again to the situation of entire graphs (and using the above formulas) we see that this equation is equivalent to
\begin{align}
u_t=-\on{div}\Big((vI-\frac{\nabla u \otimes \nabla u}{v})\nabla H\Big) \label{sdif}.
\end{align}
\begin{lemma}\label{rewrsdiff}
The equation for the graphical surface diffusion flow \eqref{sdif} can be written in the form \eqref{rewr1}.
\end{lemma}
\begin{proof}
We write
\begin{align*}
u_t=&-\on{div}(v \nabla H)+\on{div}\Big(\frac{\nabla u \otimes \nabla u}{v}\nabla H\Big)\\
=& I+II.
\end{align*}
Next we use the calculations from Lemma \ref{rewr} in order to conclude 
\begin{align*}
I=&-\Delta (v H)+\on{div} (\nabla v H)\\
=&-\Delta^2 u +\nabla^2_{ij} f^{ij}_2[u]+\nabla_i f^i_1[u].
\end{align*}
Moreover, we observe
\begin{align*}
II=& \Delta \Big(\frac{\nabla u \otimes \nabla u}{v} H\Big)-\on{div}\Big( H\nabla (\frac{\nabla u \otimes \nabla u}{v})\Big)\\
=&\nabla^2_{ij} f^{ij}_2[u]+\nabla_i f^i_1[u]
\end{align*}
and this finishes the proof of the Lemma.
\end{proof}
The fact that we can rewrite the graphical surface diffusion flow in the form \eqref{rewr1} allows to conclude the following result (compare with Theorem \ref{main}).
\begin{theorem}\label{mainsdif}
  There exists $\varepsilon>0$, $C>0$ such that for every $u_0\in
  C^{0,1}(\R^n)$ satisfying $||u_0||_{C^{0,1}(\R^n)}<\varepsilon$
  there exists an analytic solution $u\in X_\infty$ of the surface diffusion flow \eqref{sdif} with
  $u(\cdot,0)=u_0$ which satisfies $||u||_{X_\infty}\le C||u_0||_{C^{0,1}(\R^n)}$. The solution is unique in the ball $B^{X_\infty}(0,C\eps)$.

More precisely there exists $R>0$, $c>0$ such that for every $k\in \N_0$ and multiindex $\alpha\in \N_0^n$ we have the estimate
\begin{align}
  \sup_{x\in \R^n}\sup_{ t>0} |(t^{\frac{1}{4}} \nabla)^\alpha (t\partial_t)^k \nabla u(x,t)|\le c||u_0||_{C^{0,1}(\R^n)}R^{|\alpha|+k}(|\alpha|+k)!. \label{anasdif}
\end{align}  
Moreover the solution $u$ depends analytically on $u_0$. 
\end{theorem}

\subsection{Model case}
Before studying the general equation \eqref{rewr1} we study solutions of the simplified problem (which might be of independent interest) 
\begin{align}
u_t + \Delta^2 u  = f_0[u] + \nabla f_1[u]=f_M[u], \label{simply}
\end{align}  
where $f_0[u]$ and $f_1[u]$ are as in Lemma \ref{rewr}. In this case we define for every $0< T \le \infty$ the Banach space $X_{M,T}$ by
\begin{align*} 
X_{M,T}=\{u|\ \ || u ||_{X_{M,T}} := \sup_{0<t< T}|| \nabla u(t)||_{L^\infty(\R^n)} +\sup_{0<t< T}t^{\frac{1}{4}} ||  \nabla^2 u (t)||_{L^\infty(\R^n)}<\infty\}. 
\end{align*}
Moreover, we let
\[
Y_{0,M,T}=\{f_0|\ \ || f_0 ||_{Y_{0,M,T}} = \sup_{0<t<T} t^{\frac{3}{4}} ||f_0(t) ||_{L^\infty(\R^n)}<\infty\}
\]
and 
\[
Y_{1,M,T}=\{f_1|\ \ || f_1 ||_{Y_{1,M,T}} = \sup_{0<t<T} t^{\frac{1}{2}} ||f_1(t) ||_{L^\infty(\R^n)}<\infty\}. 
\]
Finally we define the norm
\[
||f||_{Y_{M,T}}=\inf \{|| f_0 ||_{Y_{0,M,T}}+|| f_1 ||_{Y_{1,M,T}}|\ \ f_0 \in Y_{0,M,T}, f_1 \in Y_{1,M,T}, \ \ f=f_0+\nabla f_1\}
\]
and the Banach space
\[
Y_{M,T}=\{ f|\ \ ||f||_{Y_{M,T}}<\infty\}.
\]

Our main goal in this subsection is to prove the following Theorem.
\begin{theorem}\label{it}
Let $0<T\le \infty$, $u \in X_{M,T}$ and $f_M[u] =f_0[u]+\nabla f_1[u]$, where $f_0[u]$ and $f_1[u]$ are as in Lemma \ref{rewr}. Then the map $F_M:C^{0,1}(\R^n)\times X_{M,T} \rightarrow X_{M,T}$, defined by $F_M(u_0,u)=Su_0+Vf_M[u]$ is analytic and we have 
\begin{align}
||F_M(u_0,u)||_{X_{M,T}}\le c(||u_0||_{C^{0,1}(\R^n)}+||u||_{X_{M,T}}^3).\label{itbm}
\end{align}
Moreover, there exists $\varepsilon_0>0$ and $q<1$ such that for all $u_0\in C^{0,1}(\R^n)$ and all $u_1,u_2\in X_{M,T}^{\varepsilon_0}=\{u\in X_{M,T}|\,\ ||u||_{X_{M,T}}<\varepsilon_0\}$ we have
\begin{align}
||F_M(u_0,u_1)-F_M(u_0,u_2)||_{X_{M,T}}\le q||u_1-u_2||_{X_{M,T}}. \label{contractionm}
\end{align}
\end{theorem}
The Theorem will be a consequence of the next two results and Lemma \ref{homog}.
\begin{lemma}\label{YX}
For every $0<T\le \infty$ the map $f_M[\cdot]=f_0[\cdot]+\nabla f_1[\cdot]:X_{M,T}\rightarrow Y_{M,T}$, where $f_0$ and $f_1$ are as in Lemma \ref{rewr}, is analytic. Moreover we have the estimates 
\begin{align}
||f_M[u]||_{Y_{M,T}}\le& c||u||_{X_{M,T}}^3\label{YX1}
\end{align} 
for every $u\in X_{M,T}$ and
\begin{align} 
||f_M[u_1]-&f_M[u_2]||_{Y_{M,T}}\nonumber \\
&\le c(||u_1||_{X_{M,T}},||u_2||_{X_{M,T}})(||u_1||_{X_{M,T}}+||u_2||_{X_{M,T}})||u_1-u_2||_{X_{M,T}} \label{YX2}
\end{align}
for all $u_1, u_2\in X_{M,T}$. 
\end{lemma}
\begin{proof}
Polynomial maps are analytic. We expand the functions of $\nabla u$ into power series. This yields a convergent power series expansion of $f_0$ and $f_1$ in the function spaces.

In order to see \eqref{YX2} we notice that for every $i\in \N$ 
\[
|v_1^{-2i}-v_2^{-2i}| \le c(|\nabla u_1|,|\nabla u_2|)v_1^{-2i}v_2^{-2i} |\nabla u_1-\nabla u_2|,
\]
where $v_j=\sqrt{1+|\nabla u_j|^2}$, $j\in \{1,2\}$.
\end{proof} 
\begin{lemma}\label{XY}
Let $0<T\le \infty$ and let $f_0+\nabla f_1=f_M\in Y_{M,T}$. Then $Vf_M\in X_{M,T}$ and we have the estimate 
\begin{align}
||Vf_M||_{X_{M,T}}\le c||f_M||_{Y_{M,T}}.\label{XY1}
\end{align}
\end{lemma}
\begin{proof}
Since the estimate \eqref{XY1} is invariant under translations and the scaling defined in \eqref{scaling} it suffices to show that
\begin{align*}
|\nabla Vf_M(0,1)|+|\nabla^2 Vf_M(0,1)|\le c ||f_M||_{Y_{M,T}}
\end{align*}
for $T>1$. Using the definition of the operator $V$ and Lemma \ref{kernel} we estimate
\begin{align*}
|\nabla Vf_M(0,1)|\le& c\int_0^1 \int_{\R^n}(|\nabla b(-y,1-s)||f_0(y,s)|\\
&+|\nabla^2 b(-y,1-s)|f_1(y,s)|)dy ds\\
\le& c||f_0||_{Y_{0,M,T}}\int_0^1 \int_{\R^n}s^{-\frac{3}{4}}\Big((1-s)^{\frac{1}{4}}+|y|\Big)^{-n-1}dy ds\\
&+c||f_1||_{Y_{1,M,T}}\int_0^1 \int_{\R^n}s^{-\frac{1}{2}}\Big((1-s)^{\frac{1}{4}}+|y|\Big)^{-n-2}dy ds\\
\le& c||f_M||_{Y_{M,T}}.
\end{align*}
Arguing similarly for $|\nabla^2 Vf_M(0,1)|$  we get
\begin{align*}
|\nabla^2 Vf_M(0,1)|\le& c\int_0^1 \int_{\R^n}\Big((1-s)^{\frac{1}{4}}+|y|\Big)^{-n-2}|f_0(y,s)|dy ds\\
&+c\int_0^1 \int_{\R^n}\Big((1-s)^{\frac{1}{4}}+|y|\Big)^{-n-3}|f_1(y,s)|dy ds.
\end{align*}
Estimating the integrals as above we get the desired bound for $|\nabla^2 Vf_M(0,1)|$.
\end{proof}

\subsection{General case}
For the general case of a solution of \eqref{rewr1} we have to include the term $\nabla^2 f_2[u]$ into our analysis. In order to do this we need to modify our function spaces while we still want them to be invariant under the scaling defined in \eqref{scaling}.
\begin{definition}\label{space}
For every $0<T\le \infty$ we define the function spaces $X_T$ and $Y_T$ by
\begin{align*}
 X_T=& \{u | \sup_{0<t<T} || \nabla u(t) ||_{L^\infty(\R^n)} +  \sup_{x\in \R^n} \sup_{0<R^4< T} 
R^{\frac{2}{n+6}}  || \nabla^2 u||_{L^{n+6}(B_R(x)\times (\frac{R^4}{2},R^4))} \\
&<\infty \} \ \ \text{and} \\
Y_T=&  Y_{0,T} + \nabla  Y_{1,T} + \nabla^2 Y_{2,T},
\end{align*}
where
\begin{align*}
 ||f_0||_{Y_{0,T}} &=    \sup_{x\in \R^n} \sup_{0<R^4< T} R^{\frac{6}{n+6}} || f_0||_{L^{\frac{n+6}{3}}(B_R(x)\times (\frac{R^4}{2},R^4))}, \\ 
||f_1||_{Y_{1,T}} &= \sup_{x\in \R^n} \sup_{0<R^4< T}  R^{\frac{4}{n+6}} || f_1||_{L^{\frac{n+6}{2}}(B_R(x)\times (\frac{R^4}{2},R^4))} \ \ \ \text{and} \\
||f_2||_{Y_{2,T}}&= \sup_{x\in \R^n} \sup_{0<R^4< T} R^{\frac{2}{n+6}}|| f_2||_{L^{n+6}(B_R(x)\times (\frac{R^4}{2},R^4))}.
\end{align*}
\end{definition}
As in the previous subsection our goal is to prove the following Theorem.
\begin{theorem}\label{ita}
Let $0<T\le \infty$, $u \in X_{T}$ and let $f[u] =f_0[u]+\nabla f_1[u]+ \nabla^2 f_2[u]$, where $f_0[u]$, $f_1[u]$ and $f_2[u]$ are as in Lemma \ref{rewr}. Then the map $F:C^{0,1}(\R^n)\times X_{T} \rightarrow X_{T}$, defined by $F(u_0,u)=Su_0+Vf[u]$ is analytic and we have \begin{align}
||F(u_0,u)||_{X_{T}}\le c(||u_0||_{C^{0,1}(\R^n)}+||u||_{X_{T}}^3).\label{itb}
\end{align}
Moreover, there exists $\varepsilon_0>0$ and $q<1$ such that for all $u_0\in C^{0,1}(\R^n)$ and all $u_1,u_2\in X_{T}^{\varepsilon_0}=\{u\in X_{T}|\,\ ||u||_{X_{T}}<\varepsilon_0\}$ we have
\begin{align}
||F(u_0,u_1)-F(u_0,u_2)||_{X_{T}}\le q||u_1-u_2||_{X_{T}}. \label{contraction}
\end{align}
\end{theorem}
The Theorem will be a consequence of the following two results and Lemma \ref{homog}.
\begin{lemma}\label{yx}
For every $0<T\le \infty$ the operator $f[\cdot]=(f_0+\nabla f_1+\nabla^2 f_2)[\cdot]:X_T\rightarrow Y_T$, where $f_0$, $f_1$ and $f_2$ are as in Lemma \ref{rewr}, is analytic and we have the estimates
\begin{align}
 ||f[u]||_{Y_T} \le c||u||_{X_T}^3 \label{estyx}
\end{align}
for all $u\in X_T$ and
\begin{align}
||f[u_1]-&f[u_2]||_{Y_T} \le c(||u_1||_{X_T},||u_2||_{X_T})(||u_1||_{X_T}+||u_2||_{X_T})||u_1-u_2||_{X_T} \label{estyxa}
\end{align}
for all $u_1,u_2 \in X_T$.
\end{lemma}
\begin{proof}
The proof is the same as the one for Lemma \ref{YX}.
\end{proof}
\begin{lemma}\label{xy}
Let $0<T\le \infty$ and $f_0+\nabla f_1+\nabla^2 f_2=f\in Y_T$. Then $Vf\in X_T$ and we have
\begin{align}
||Vf||_{X_T}\le c||f||_{Y_T}.\label{xy1}
\end{align}
\end{lemma}
\begin{proof}
Since the estimate \eqref{xy1} is invariant under translations and the scaling defined in \eqref{scaling} it is enough to show that
\begin{align*}
|\nabla Vf(0,1)|+||\nabla^2 Vf||_{L^{n+6}(B_1(0)\times (\frac{1}{2},1))} \le c||f||_{Y_T}
\end{align*}
for some $T>1$.

By definition we have
\[ 
\nabla Vf(0,1) = \int_0^1\int_{\R^n} \nabla b(x,1-t) f(x,t)  dx dt
\]
and, with $Q= B_1(0) \times [\frac{1}{2},1)$ and $Q' = B_1(0) \times (0,\frac12)$, we decompose
\begin{align*}
|\nabla Vf(0,1)|\le& \Big|\int_Q \nabla b(x,1-t) f(x,t)  dx dt\Big|+\Big|\int_{\R^n\times (0,1) \backslash Q} \nabla b(x,1-t) f(x,t)  dx dt\Big|\\
=&I+II.
\end{align*}
Now we estimate term by term. We start with $I$.
\begin{align*} 
I\le& \left| \int_Q  \Big( \nabla b(x,1-t) f_0 (x,t)- \nabla^2 b(x,1-t) f_1(x,t) + \nabla^3 b(x,1-t) f_2(x,t) \Big) dx dt \right|  \\
\le& 
\Vert \nabla b \Vert_{L^{\frac{n+6}{n+3}}(Q')}  \Vert f_0 \Vert_{L^\frac{n+6}3(Q)} + \Vert \nabla^2 b \Vert_{L^{\frac{n+6}{n+4}}(Q')} \Vert f_1 \Vert_{L^{\frac{n+6}2}(Q)}  \\
&+ \Vert \nabla^3 b \Vert_{L^{\frac{n+6}{n+5}}(Q') } \Vert f_2 \Vert_{L^{n+6}(Q)}  \\
\le& c||f||_{Y_T},
\end{align*} 
where we used the fact that
\begin{align}
||\nabla b||_{L^{\frac{n+6}{n+3}}(\R^n\times [0,1])}+||\nabla^2 b||_{L^{\frac{n+6}{n+4}}(\R^n\times [0,1])}+||\nabla^3 b||_{L^{\frac{n+6}{n+5}}(\R^n\times [0,1])}\le c, \label{estb} 
\end{align}
which is a consequence of the estimate \eqref{kernel2}.

Integrating by parts and using \eqref{decay} we get
\begin{align*}
II\le& c\sum_{m=0}^\infty \sum_{y\in \Z^n}  \int_{2^{-m-1}}^{2^{-m}} \int_{B_1(y)} e^{-c_1|x|} (|f_0|+|f_1|+|f_2|)(x,t)  dx dt\\ 
\le&  c \sum_{m=0}^\infty \sup_{y\in \Z^n}  \int_{B_1(y) \times (2^{-m-1},2^{-m})} (|f_0|+|f_1|+ |f_2|) dx dt. 
\end{align*}
Now we claim that there exists a number $\gamma <1$ such that
\begin{align}   
\int_{B_1(0) \times (2^{-m-1}, 2^{-m})}(|f_0|+|f_1|+|f_2|)(x,t) dx\, dt 
 \le c \gamma^m \Vert f \Vert_{Y_T}. \label{gamma} 
\end{align} 
Using the translation invariance this claim then implies that
\begin{align*}
II\le  c \Vert f \Vert_{Y_T},
\end{align*}
which, combined with the above estimate, shows that
\begin{align*}
|\nabla Vf(0,1)|\le c\Vert f \Vert_{Y_T}
\end{align*}
and finishes the proof of the first part of the estimate \eqref{xy1}.

To prove \eqref{gamma} we cover the set $B_1(0) \times (2^{-m-1}, 2^{-m})$ by approximately $2^{\frac{nm}{4}}$ cylinders of the form $Q_m(y):= B_{2^{-\frac{m}{4}}}(y) \times (2^{-m-1},2^{-m})$. By H\"older's inequality we get
\begin{align*}  
\sum_{j=0}^2 \Vert f_j \Vert_{L^1(Q_m(y))}
\le&  c  (  2^{ \frac{6m-m(n+4)(n+3)}{4(n+6)}} + 2^{  \frac{4m- m(n+4)^2}{4(n+6)}} +     2^{\frac{2m-m(n+4)(n+5)}{4(n+6)}}  )   \Vert f\Vert_{Y_T} 
\\ \le &  c 2^{-\frac{m}{4}-\frac{mn}{4}}     \Vert f\Vert_{Y_T} ,        
\end{align*}  
which implies \eqref{gamma} with $\gamma = 2^{-\frac14}$.

Using the same arguments as in the estimate for the term $II$ above we get the pointwise bound 
\begin{align*} 
\sup_{(x,t)\in Q}| & \int_{\R^n\times (0,1) \backslash B_2(0)\times (\frac14,1)}  \Big(\nabla^2 b(x-y,t-s) f_0(y,s)-\nabla^3 b(x-y,t-s) f_1(y,s)\\
&+\nabla^4 b(x-y,t-s) f_2(y,s)\Big)  dy ds |\le c \Vert f \Vert_{Y_{T}}.
\end{align*}
Therefore it remains to show the estimate for the $L^{n+6}$-norm of $\nabla^2 Vf$ on $Q$ for functions $f_0$, $f_1$ and $f_2$ whose support is contained in $B_2(0)\times (\frac14,1)$. 

In this situation the estimate for $\nabla^2 V(\nabla^2 f_2)$ follows immediately from the fact that
\begin{align} 
\Vert \nabla^2 u \Vert_{L^{n+6}(\R^n \times \R^+)} 
\le c \Vert f_2 \Vert_{L^{n+6}(\R^n \times \R^+)} \label{CZ}
\end{align}
for all solutions $u$ of 
\begin{align} 
u_t + \Delta^2 u = \nabla_{ij}^2 f^{ij}_2, \qquad u(\cdot,0)=0. \label{u}
\end{align}
\eqref{CZ} can be seen as follows: 
Multiplying the equation \eqref{u} by $u$ and integrating by parts, we get with the help of H\"older's inequality
\begin{align*} 
|| \Delta u ||^2_{L^2(\R^n\times \R^+)} \le c|| f_2 ||_{L^2(\R^n\times \R^+)}||\nabla^2 u ||_{L^2(\R^n\times \R^+)}. 
\end{align*}
Integrating by parts again and interchanging derivatives yields
\begin{align*} 
|| \nabla^2 u ||_{L^2(\R^n\times \R^+)} \le c|| f_2 ||_{L^2(\R^n\times \R^+)}. 
\end{align*}
Hence the operator which maps $f_2^{ij}$ to $\partial^2_{kl} u $ is a continuous and linear operator $T_{ijkl}$ from $L^2$ to $L^2$, which has an integral kernel given by $\partial^4_{ijkl} b$. We equip $\R^n\times \R$ with the metric 
\[ d((x,t),(y,s)) = \max \{ |x-y|, |t-s|^{1/4} \}. \]
and we let $m^{n+1}$ be the Lebesgue measure. The triple $(\R^n \times \R, d, m^{n+1})$ is a space of homogeneous type and $T_{ijkl}$ is a singular integral operator in this non Euclidean setting. As a consequence of this we get for every $1<p< \infty$ that (see for example \cite{stein93})
\[  
|| T_{ijkl} f ||_{L^p(\R^n\times \R^+)}  \le c \frac{p^2}{p-1} || f ||_{L^p(\R^n\times \R^+)}
\]
and this shows \eqref{CZ} by choosing $p=n+6$. 

In order to estimate the $L^{n+6}$-norm of $\nabla^2 Vf_0$ and $\nabla^2 V\nabla f_1$ we recall the Young inequality 
\begin{align*}
|| f*g ||_{L^m(\R^n \times \R^+)} \le c || f ||_{L^p(\R^n \times \R^+)} || g ||_{L^{q}(\R^n \times \R^+)}, 
\end{align*}
where
\begin{align*}
 0 \le p,q,m \le \infty\ \ \ \text{and}\ \ \  \frac1p + \frac1q = 1+ \frac1m.
\end{align*}
Applying this inequality with $m=n+6$, $p=\frac{n+6}{3}$, $q=\frac{n+6}{n+4}$, respectively $m=n+6$, $p=\frac{n+6}{2}$, $q=\frac{n+6}{n+5}$ and using \eqref{estb} we therefore get
\begin{align*}
||\nabla^2 Vf_0||_{L^{n+6}(\R^n \times (0,1))}+&||\nabla^2 V\nabla f_1||_{L^{n+6}(\R^n \times (0,1))}\\
\le& c(||f_0||_{L^{\frac{n+6}{3}}(\R^n \times \R^+)}+|| f_1||_{L^{\frac{n+6}{2}}(\R^n \times \R^+)}).
\end{align*}
Since the support of $f_0$ and $f_1$ is contained in $B_2(0)\times (\frac14,1)$ this estimate completes the proof of the Lemma. 
\end{proof}

\subsection{Proof of Theorem \ref{main}}
\begin{proof}
For every $0<T \le \infty$ and $u_0\in C^{0,1}(\R^n)$ we define the operator $F_{u_0}:X_T \rightarrow X_T$ by
\begin{align}
F_{u_0}(u)=F(u_0,u)=Su_0+Vf[u], \label{oper}
\end{align}
where $f[u]=f_0[u]+\nabla f_1[u]+\nabla^2 f_2[u]$ and $f_0[u]$, $f_1[u]$ and $f_2[u]$ are as in Lemma \ref{rewr}. From Theorem \ref{ita} and the Banach fixed point theorem we get that there exist $\delta_1,\delta_2>0$ such that for all $u_0\in C^{0,1}(\R^n)$ with $||u_0||_{C^{0,1}(\R^n)}<\delta_1$ the map $F_{u_0}$ has a unique fixed point $u\in X_T^{\delta_2}$ ($X_T^{\delta_2}$ is defined in Theorem \ref{ita}). Moreover $u$ depends on $u_0$ in a Lipschitz continous way. Thus $u$ is the unique global solution of \eqref{rewr1} we were looking for.

Next we show that $u$ depends analytically on $u_0$. From Theorem \ref{ita} we get that for every $0<T\le \infty$ the map $G:C^{0,1}(\R^n)\times X_T \rightarrow X_T$, defined by
\[
G(u_0,u)=u-Su_0-Vf[u]=u-F(u_0,u),
\] 
is analytic, $G(0,0)=0$ and
\begin{align*}
D_uG(0,0)=id.
\end{align*}
Combining all these facts we can apply the (analytic) implicit function theorem (see for example \cite{deimling}) to get the existence of balls $B_\varepsilon(0)\subset C^{0,1}(\R^n)$, $B_\gamma(0) \subset X_T$ and a unique analytic map $A:B_\varepsilon(0)\subset C^{0,1}(\R^n)\rightarrow B_\gamma(0)\subset X_T$ such that $A(0)=0$ and $G(u_0,A(u_0))=0$ for all $u_0 \in B_\varepsilon(0)$. Moreover $G(u_0,u_1)=0$ if and only if $u_1=A(u_0)$. From the above considerations we conclude that for $\delta=\min\{\delta_2,\gamma\}$ there exists a unique solution $u\in X_T^{\delta}$ of \eqref{rewr1} which depends analytically on the initial data $u_0$.

It remains to prove that $u(x,t)$ is analytic in $x$ and $t$ for every $x\in \R^n$ and $0<t<\infty$. In order to do this we let $T<\infty$ and we define for $\eps_1,\eps_2>0$ small an operator $\tilde{G}:D_{\eps_1}(0)\times (1-\eps_2,1+\eps_2)\times C^{0,1}(\R^n) \times X_T \rightarrow X_T$, where $D_{\eps_1}(0)\subset \R^n$, by
\begin{align}
\tilde{G}(a,\tau,u_0,u)=u-Su_0-V\tilde{f}_{a,\tau}[u],\label{tildeF}
\end{align} 
where
\begin{align*}
\tilde{f}_{a,\tau}[u]=\tau f[u]+(1-\tau)\Delta^2 u-a\nabla u.
\end{align*}
Note that $\tilde{f}_{0,1}[u]=f[u]$. By using Lemma \ref{kernel} it is easy to see that $||Va\nabla u||_{X_T}\le c|a|T^{\frac{3}{4}} ||u||_{X_T}$. Moreover, by defining $\tilde{f}[u]=\tau f[u]+\Delta \big((1-\tau)\Delta u\big)$, we get that Theorem \ref{ita} remains valid for $\tilde{f}[u]$ and we have the estimate
\begin{align*}
||\tilde{G}(a,\tau,u_0,u)||_{X_T}\le c\Big(||u_0||_{C^{0,1}(\R^n)}+||u||_{X_T}(1+|a|T^{\frac{3}{4}}+|1-\tau|+||u||_{X_T}^2)\Big).
\end{align*}
Since $\tilde{G}(0,1,0,0)=0$ and $D_u \tilde{G}(0,1,0,0)=id$, another application of the implicit function theorem gives the existence and uniqueness of an analytic map $\tilde{A}: D_{\tilde{\varepsilon}}(0) \times (1-\tilde{\varepsilon},1+\tilde{\varepsilon}) \times B_{\tilde{\varepsilon}}(0) \rightarrow X_T$ such that $\tilde{G}(a,\tau,u_0,\tilde{A}(a,\tau,u_0))=0$ and therefore
\begin{align*}
\tilde{A}(a,\tau,u_0)=Su_0+V\tilde{f}[\tilde{A}(a,\tau,u_0)].
\end{align*}
Next we let $\overline{\varepsilon}<\min\{\varepsilon,\tilde{\varepsilon}\}$ and we observe from the above uniqueness results that $A(u_0)(x-at,\tau t)=\tilde{A}(a,\tau,u_0)$ since $A(u_0)(\cdot,0)=u_0=\tilde{A}(a,\tau,u_0)(\cdot,0)$ and $\tilde{G}(a,\tau,u_0,A(u_0)(x-at,\tau t))=0$. 

Hence we get that $A(u_0)(x-at,\tau t)$ depends analytically on $a$ and $\tau$. Since for finite $t$ we moreover have that
\begin{align*}
\frac{\partial}{\partial a} A(u_0)(x-at,\tau t)|_{(a,\tau)=(0,1)}=& -t \nabla A(u_0)(x,t),\\
\frac{\partial}{\partial \tau} A(u_0)(x-at,\tau t)|_{(a,\tau)=(0,1)}=& t \partial_t A(u_0)(x,t),
\end{align*}
with similar formulas for higher and mixed derivatives, we conclude that $A(u_0)$ and therefore also $u$ is analytic in space and time for all $x\in \R^n$ and all $0<t<\infty$. The estimate \eqref{ana} (resp. \eqref{anasdif}) now follows from a scaling argument and the above formula for the derivatives of $u$.
\end{proof}

\subsection{Self-similar solutions}

In this subsection we use Theorem \ref{main} in order to show the existence of self-similar solutions of the Willmore and surface diffusion flow for graphs. More precisely we show the existence of homothetically expanding solutions. Since the arguments for both flows are identitical we restrict ourselves to the the situation of the Willmore flow. 

We consider self-similar initial data $u_0$, i.e. $u_0$ which satisfy
\begin{align}
u_0(x)=\frac{1}{\lambda}u_0(\lambda x)\ \ \text{for any}\ \ \lambda>0 \ \ \text{and}\ \ x\in \R^n. \label{self}
\end{align}
Hence $\Sigma_0=\text{graph}(u_0)$ is a cone with vertex $0$. If we
assume that $||u_0||_{C^{0,1}(\R^n)}<\varepsilon$, where $\varepsilon$
is as in Theorem \ref{main}, we get from Theorem \ref{main} the
existence of a unique analytic solution $u\in X_\infty$ of \eqref{wflow} with
initial condition $u(\cdot,0)=u_0$. Next, if we define
$u_{0,\lambda}(x)=\frac{1}{\lambda}u_0(\lambda x)$, we get that
$||u_{0,\lambda}||_{C^{0,1}(\R^n)}=||u_0||_{C^{0,1}(\R^n)}<\varepsilon$
and hence $u_\lambda(x,t)=\frac{1}{\lambda}u(\lambda x,\lambda^4 t)$
is the unique analytic solution of \eqref{wflow} in $X_\infty$ with
$u_\lambda(\cdot, 0)=u_{0,\lambda}$. Since by \eqref{self} we have
that $u_0=u_{0,\lambda}$ we get that for any self-similar initial data
$u_0$ with $||u_0||_{C^{0,1}(\R^n)}<\varepsilon$ there exists a unique
analytic solution of \eqref{wflow} which satisfies
\begin{align}
u(x,t)=\frac{1}{\lambda}u(\lambda x,\lambda^4 t)\label{self2}
\end{align}
for any $x\in \R^n$, $t>0$ and $\lambda>0$. Defining $\lambda=t^{-\frac{1}{4}}$ and $\Psi(y)=u(y,1)$ (note that $\Psi$ is analytic) we get that 
\begin{align}
u(x,t)=t^{\frac{1}{4}} \Psi(xt^{-\frac{1}{4}}).\label{self3}
\end{align}
Moreover $\Psi$ satisfies the elliptic equation
\begin{align}
\Delta^2 \Psi +\frac{1}{4}\Psi-\frac{y}{4}\cdot \nabla \Psi=f[\Psi],\label{self4}
\end{align}
where $f$ is as in Lemma \ref{rewr}. Combining all these facts we get the following Theorem.
\begin{theorem}\label{selfa}
  There exists $\varepsilon>0$, $C>0$ such that if $u_0\in C^{0,1}(\R^n)$ is
  self-similar with $||u_0||_{C^{0,1}(\R^n)}<\varepsilon$, then there
  exists a global analytic and self-similar solution $u\in X_\infty$ of the
  Willmore flow \eqref{wflow} which satisfies the estimates $||u||_{X_\infty}\le C||u_0||_{C^{0,1}(\R^n)}$ and \eqref{ana}. The solution is unique in the ball $B^{X_\infty}(0,C\eps)$. Moreover $u$ can be written in the form \eqref{self3},
  where $\Psi$ is an analytic solution of \eqref{self4}.
\end{theorem}

\section{Ricci-DeTurck flow}
\setcounter{equation}{0}

On a manifold $M^n$ with a family of Riemannian metrics $g(t)$ the Ricci flow is given by
\begin{align}
\partial_t g=&-2\text{Ric}(g)\ \ \ \text{in}\ \ M^n\times (0,T)\ \ \ \text{and}\nonumber\\
g(\cdot,0)=&g_0,\label{Ricci}
\end{align}
where $\text{Ric}(g)$ denotes the Ricci curvature of $g$ and $g_0$ is some metric on $M^n$. In this section we are interested in a closely related flow, the so called Ricci-DeTurck flow for a family of Riemannian metrics $g(t)$ on $\R^n$. This flow is given by (see for example \cite{guenther02})
\begin{align}
\partial_t g=&-2\text{Ric}(g)-P_\delta(g)\ \ \ \text{in}\ \ \R^n\times (0,T)\ \ \ \text{and}\nonumber\\
g(\cdot,0)=&g_0,\label{RicciDe}
\end{align}
where $\delta$ is the euclidean metric and 
\[
P_\delta(g)=-2d_g^\star d_g(G(g,\delta)),
\]
where $G(g,\delta)=\delta-\frac{n}{2}g$, $d_g:h\rightarrow
d_gh=-g^{ij}\nabla_i h_{jk} dx^k$ ($d_g$ maps symmetric covariant
two-tensors onto one-forms) and $d_g^\star:\omega\rightarrow d^\star_g
\omega=\frac{1}{2}(\nabla_i \omega_j+\nabla_j \omega_i)dx^i \otimes
dx^j$ ($d_g^\star$ is the adjoint operator of $d_g$ with respect to
the $L^2$ inner product and therefore it maps one-forms onto symmetric
covariant two-tensors). 

The Ricci-DeTurck flow was introduced by DeTurck \cite{deturck83} in order to give a short proof for the short-time existence of the Ricci flow on compact manifolds. DeTurck achieved this goal by showing that the flows are equivalent (see also \cite{shi89}) and that \eqref{RicciDe} is a strictly parabolic system for which the general short-time existence theory can be applied. 

In local coordinates the Ricci-DeTurck flow \eqref{RicciDe} can be written as (see \cite{shi89})
\begin{align}
  \partial_t g_{ij} =& g^{ab} \nabla_{a}\nabla_b g_{ij}+\frac{1}{2}g^{ab} g^{pq}\Big( \nabla_i g_{pa} \nabla_j g_{qb}+2\nabla_a g_{jp} \nabla_q g_{ib}-2\nabla_a g_{jp} \nabla_b g_{iq} \nonumber \\
  &-2\nabla_j g_{pa} \nabla_b g_{iq}-2\nabla_i g_{pa} \nabla_b g_{jq}\Big), \label{ricci1}
\end{align}
where all the derivatives are taken with respect to the euclidean background metric. We can rewrite this system as follows
\begin{align}
(\partial_t -\Delta) h_{ij} =& \nabla_a \Big( \big( (\delta+h)^{ab}-\delta^{ab}\big)\nabla_b h_{ij}\Big)-\nabla_a (\delta+h)^{ab} \nabla_b h_{ij}\nonumber \\
&+\frac{1}{2}(\delta+h)^{ab} (\delta+h)^{pq}\Big( \nabla_i h_{pa} \nabla_j h_{qb}+2\nabla_a h_{jp} \nabla_q h_{ib}-2\nabla_a h_{jp} \nabla_b h_{iq}\nonumber \\
&-2\nabla_j h_{pa} \nabla_b h_{iq}-2\nabla_i h_{pa} \nabla_b h_{jq}\Big)\nonumber \\
=& R_0[h]+\nabla R_1[h],\label{ricci}
\end{align}
where $h=g-\delta$, $\Big( (\delta+h)^{ab}\Big)=\Big( (\delta+h)_{ab}\Big)^{-1}$ and
\begin{align*}
R_0[h]=& \frac{1}{2}(\delta+h)^{ab} (\delta+h)^{pq}\Big( \nabla_i h_{pa} \nabla_j h_{qb}+2\nabla_a h_{jp} \nabla_q h_{ib}-2\nabla_a h_{jp} \nabla_b h_{iq} \\
&-2\nabla_j h_{pa} \nabla_b h_{iq}-2\nabla_i h_{pa} \nabla_b h_{jq}\Big)-\nabla_a (\delta+h)^{ab} \nabla_b h_{ij}\ \ \ \text{and}\\
\nabla R_1[h]=&\nabla_a \Big( \big( (\delta+h)^{ab}-\delta^{ab}\big)\nabla_b h_{ij}\Big).
\end{align*}

We note that the Ricci-DeTurck flow is invariant under the scaling ($\lambda>0$)
\begin{align}
h_\lambda(x,t)=h(\lambda x,\lambda^2 t). \label{scalingR}
\end{align}

We remark that in the rest of this section all norms are taken with respect to the euclidean metric $\delta$.

For all $0<T\le \infty$ we define the function spaces 
\begin{align*}
X_T&=\{h |\,\ ||h||_X=\sup_{0<t<T}||h(t)||_{L^\infty(\R^n)}\\
+&\sup_{x\in \R^n}\sup_{0<R^2<T}
\left( R^{-\frac{n}{2}} \Vert \nabla h \Vert_{L^2(B_R(x) \times (0,R^2))} + R^{\frac{2}{n+4}} \Vert \nabla h \Vert_{L^{n+4}(B_R(x) \times (\frac{R^2}{2},R^2))}   \right) \\
<&\infty \}\ \ \text{and} \\ 
Y_T&= Y_T^0 + \nabla Y_T^1,
\end{align*}  
where 
\begin{align*}
 \Vert f \Vert_{Y_T^0} =&   \sup_{x\in \R^n} \sup_{0< R^2 < T} 
\left(R^{-n} \Vert f \Vert_{L^1(B_R(x) \times(0,R^2))} + R^{\frac4{n+4}} 
\Vert f \Vert_{L^{\frac{n+4}2}(B_R(x) \times (\frac{R^2}{2}, R^2))}\right), \\
\Vert f \Vert_{Y^1_T} =&  \sup_{x\in \R^n} \sup_{0< R^2 < T} 
\left(R^{-\frac{n}{2}} \Vert f \Vert_{L^2(B_R(x) \times(0,R^2))} + R^{\frac2{n+4}} 
\Vert f \Vert_{L^{n+4}(B_R(x) \times (\frac{R^2}{2}, R^2))}\right) .
\end{align*} 
Note that these spaces are both invariant under the scaling defined in
\eqref{scalingR}. 

From the definition of the spaces $X_T$, $Y_T$ and
the expressions for $R_0[h]$, $R_1[h]$ we directly get the following Lemma.
\begin{lemma}\label{first}
For every $0<T\le \infty$ and every $0<\gamma<1$ the operator $R_0[\cdot]+\nabla R_1[\cdot]:  X_T^{\gamma}=\{ h\in X_T|\,\ ||h||_{X_T}\le \gamma\}\rightarrow Y_T$ is analytic and we have the estimate
\begin{align}
||R_0[h]+\nabla R_1[h]||_{Y_T} \le c(\gamma) ||h||^2_{X_T}\label{firsta}
\end{align}
for all $h\in  X_T^{\gamma}$ and 
\begin{align}
||&R_0[h_1]-R_0[h_2]+\nabla( R_1[h_1]-R_1[h_2])||_{Y_T} \nonumber \\
&\le c(\gamma)(||h_1||_{X_T}+||h_2||_{X_T}) ||h_1-h_2||_{X_T}\label{firstaa}
\end{align}
for all $h_1,h_2 \in  X_T^{\gamma}$.
\end{lemma}
Moreover we have
\begin{lemma}\label{firsti}
Let $0<T\le \infty$ and $R=R_0+\nabla R_1\in Y_T$. Then every solution $h$ of $(\partial_t-\Delta) h=R$ with $h(\cdot,0)=h_0\in L^\infty(\R^n)$ is in $X_T$ and we have the estimate
\begin{align}
||h||_{X_T}\le c(||h_0||_{L^\infty(\R^n)}+||R||_{Y_T}).\label{firstai}
\end{align}
\end{lemma}
\begin{proof}
First of all we note that by Lemma \ref{homog2nd} we can assume without loss of generality that $h_0=0$. Next we note that the estimate \eqref{firstai} is invariant under translations and the scaling defined in \eqref{scalingR}. Hence it suffices to show that
\begin{align*}
|h(0,1)|+||\nabla h||_{L^2(B_1(0)\times(0,1))}+||\nabla h||_{L^{n+4}(B_1(0)\times(\frac{1}{2},1))}\le c||R||_{Y_T},
\end{align*}
for some $T>1$.

The estimate for the $L^\infty$-norm of $h$ follows from arguments similar to the ones used in the proof of Lemma \ref{xy}. More precisely, we decompose
\begin{align*}
|h(0,1)| \le& |\int_Q \Phi(x,1-t)R(x,t)dx dt|+|\int_{\R^n\times (0,1) \backslash Q} \Phi(x,1-t)R(x,t)dx dt|\\
\le& I+II,
\end{align*}
where we let again $Q=B_1(0)\times [\frac12,1)$ and $Q'=B_1(0) \times (0,\frac12)$.

Now we estimate $I$ by
\begin{align*}
I\le& ||\Phi||_{L^{\frac{n+4}{n+2}}(Q')}||R_0||_{L^{\frac{n+4}{2}}(Q)}+||\nabla \Phi||_{L^{\frac{n+4}{n+3}}(Q')}||R_1||_{L^{n+4}(Q)}\le c||R||_{Y_T},
\end{align*}
where we used that
\begin{align*}
||\Phi||_{L^{\frac{n+4}{n+2}}(\R^n\times [0,1])}+||\nabla \Phi||_{L^{\frac{n+4}{n+3}}(\R^n\times [0,1])}\le c,
\end{align*}
which is a consequence of \eqref{kernel2b}.

Integration by parts and \eqref{phi} yield
\begin{align*}
II \le& c\sum_{y\in \Z^n} \int_{0}^1 \int_{B_1(y)} e^{-c_1|x|}(|R_0(x,t)|+|R_1(x,t)|)dx dt\\
\le& c||R||_{Y_T}.
\end{align*}
Combining the above estimates we conclude
\begin{align*}
|h(0,1)|\le c||R||_{Y_T}.
\end{align*}

In order to estimate the $L^2$-norm of $\nabla h$ we multiply the equation $(\partial_t-\Delta)h=R_0+\nabla R_1$ with $\eta^2 h$, where $\eta$ is defined as in the proof of Lemma \ref{homog2nd}, integrate by parts and use Young's inequality to get
\begin{align*}
\partial_t \int_{\R^n} \eta^2 |h|^2+\int_{\R^n}\eta^2 |\nabla h|^2 \le c\int_{B_2(0)}(|h|^2+|h||R_0|+|R_1|^2).
\end{align*}
Integrating in time and using the pointwise estimate for $h$ yields
\begin{align*}
||\nabla h||_{L^2(B_1(0)\times(0,1))} \le c||R||_{Y_T}.
\end{align*}

Hence it remains to estimate the $L^{n+4}$-norm of $\nabla h$ on $Q$. Arguing as in the estimate for $II$ above we have 
\begin{align*}
\sup_{(x,t)\in Q}|& \int_{\R^n\times(0,1)\backslash B_2(0)\times (\frac14,1)} \Big( \nabla \Phi(x-y,t-s)R_0(y,s)\\
&- \nabla^2 \Phi(x-y,t-s) R_1(y,s) \Big)dy ds|\le c||R||_{Y_T}
\end{align*}
and hence we can assume that the support of $R_0$ and $R_1$ is contained in $B_2(0)\times (\frac14,1)$. In this situation we can use an argument involving singular integrals and the Young inequality as in the proof of Lemma \ref{xy} to finish the estimate for the $L^{n+4}$-norm of $\nabla h$.
\end{proof}

Since we know that the linearization of the operator
\begin{align*}
\partial_t+2\text{Ric}(\cdot)+P_\delta(\cdot)
\end{align*} 
at $g=\delta$ is $\partial_t-\Delta$ (see e.g. \cite{guenther02}) we
can argue as in the proof of Theorem \ref{main} to get the following result
\begin{theorem}\label{riccimain}
There exists $\varepsilon>0$, $C>0$ such that for every metric $g_0\in
  L^\infty(\R^n)$ satisfying $||g_0-\delta||_{L^\infty(\R^n)}<\varepsilon$
  there exists a global analytic solution $g\in \delta+X_\infty$ of the Ricci-DeTurck flow \eqref{RicciDe} with $g(\cdot,0)=g_0$ and $||g-\delta||_{X_\infty}\le C||g_0-\delta||_{L^\infty(\R^n)}$. The solution is unique in the ball $B^{X_\infty}(\delta,C\eps)=\{g|\ \ ||g-\delta||_{X_\infty}\le C\eps\}$.

More precisely there exists $R>0$, $c>0$ such that for every $k\in \N_0$ and every multiindex $\alpha\in \N_0^n$ we have the estimate
\begin{align}
  \sup_{x\in \R^n}\sup_{ t>0} |(t^{\frac{1}{2}} \nabla)^\alpha (t\partial_t)^k (g-\delta)(x,t)|\le c||g_0-\delta||_{L^\infty(\R^n)}R^{|\alpha|+k}(|\alpha|+k)!. \label{anaRD}
\end{align}  
Moreover the solution $g$ depends analytically on $g_0$. 
\end{theorem}
This result improves Theorem $1.2$ of \cite{schnuerer08} since the solution we construct is unique in $X_\infty$, analytic in $x$ and $t$ and the initial metric $g_0$ is only assumed to be in $L^\infty(\R^n)$. We like to remark that on general complete manifolds local solutions of \eqref{RicciDe} have been constructed by Simon \cite{simon02} for initial metrics $g_0\in C^0$ which are close to a smooth metric with bounded sectional curvature (see also \cite{simon08}). 

The relation between the solution of \eqref{RicciDe} constructed in Theorem \ref{riccimain} and a solution of the Ricci flow is illustrated in the following remark.
\begin{bem}\label{equivalence}
Let $g_0$ be a smooth initial metric satisfying $||g_0-\delta||_{L^\infty(\R^n)}<\varepsilon$ and let $g\in \delta+X_\infty$ be the analytic solution  of \eqref{RicciDe} constructed in Theorem \ref{riccimain}. It is shown in \cite{schnuerer08} that there exists a smooth family of diffeomorphisms $\varphi:\R^n\times[0,\infty)\rightarrow \R^n$ with $\varphi(\cdot,0)=id$ such that the family of metrics $\tilde{g}(x,t)=(\varphi(x,t))^\star g(x,t)$ is a solution of the Ricci flow \eqref{Ricci} with initial data $g_0$.
\end{bem}

\section{Mean curvature flow}
\setcounter{equation}{0}
Let $M^{n}$ be a $n$-dimensional orientable manifold and let $F_0:M \rightarrow \R^{n+m}$ ($m\in \N$) be an immersion. We say that the family of immersions $F:M\times [0,T)\rightarrow \R^{n+m}$ solves the mean curvature flow with initial condition $F_0$ if
\begin{align}
\partial_t F=& \bf{H}\ \ \ \rm \text{on}\ \ M\times(0,T)\ \ \ \text{and}\nonumber\\
F(\cdot,0)=& F_0, \label{meang}
\end{align}
where $\bf{H}$$(x,t)$ is the mean curvature vector of $M_t=F(M,t)$ at $F(x,t)$. Here we are interested in the case $M=\R^n$ and where $F_0(x)=(x,f_0(x))$, $f_0:\R^n\rightarrow \R^m$, is the graph of $f_0$ (entire graph). More precisely we consider $f_0\in C^{0,1}(\R^n,\R^m)$ and we assume that the Lipschitz norm of $f_0$ is "small". Then we construct solutions $f:\R^n\times [0,\infty) \rightarrow \R^m$ of the parabolic system
\begin{align}
\partial_t f =& g^{ij} \frac{\partial^2 f}{\partial x^i \partial x^j},\nonumber\\
f(\cdot,0)=&f_0, \label{meanh}
\end{align}
where $g_{ij}=\delta_{ij}+\langle \frac{\partial f}{\partial x^i},  \frac{\partial f}{\partial x^j} \rangle$. 

For $m=1$ we calculate
\begin{align*}
g^{ij}=\delta_{ij}-\frac{\nabla_i f \nabla_j f}{1+|\nabla f|^2}
\end{align*}
and therefore we have
\begin{align*}
\partial_t f=\sqrt{1+|\nabla f|^2} \,\ \text{div}\Big(\frac{\nabla f}{\sqrt{1+|\nabla f|^2}}\Big) 
\end{align*}
and hence we recover the well-known equation for the mean curvature flow for graphs in codimension one. 

Concerning the relation between solutions of the equations \eqref{meang} and \eqref{meanh} it was shown in \cite{wang04}, Proposition $2.2$, that for every graphical solution $F$ of \eqref{meang} there exists a family of diffeomorphisms $r:\R^n\times [0,\infty)\rightarrow \R^n$ such that $\tilde{F}(x,t)=F(r(x,t),t)$ can be written as $\tilde{F}(x,t)=(x,f(x,t))$ and $f$ is a solution of \eqref{meanh}. Conversely, if $f$ is a solution of \eqref{meanh}, then $\tilde{F}(x,t)=(x,f(x,t))$ is a solution of \eqref{meang}.  

Next we note that \eqref{meanh} can equivalently be written as
\begin{align}
(\partial_t-\Delta) f =& (g^{ij}-\delta^{ij}) \frac{\partial^2 f}{\partial x^i \partial x^j}=:M[f],\nonumber\\
f(\cdot,0)=&f_0 \label{mean}
\end{align}
and this system is invariant under the scaling ($\lambda>0$)
\begin{align}
f_\lambda(x,t)=\frac{1}{\lambda}f(\lambda x,\lambda^2 t). \label{scalingmean}
\end{align}
For every $0<T\le \infty$ we define the function spaces
\begin{align*}
X_T=\{f |\,\ ||f||_{X_T}=&\sup_{0<t<T}||\nabla f(t)||_{L^\infty(\R^n)}\\
&+\sup_{x\in \R^n}\sup_{0<R^2<T}R^{\frac{2}{n+4}} \Vert \nabla^2 f \Vert_{L^{n+4}(B_R(x) \times (\frac{R^2}{2},R^2))}<\infty \}\ \ \text{and} \\ 
Y_T= \{ g|\,\ ||g||_{Y_T} =&  \sup_{x\in \R^n} \sup_{0< R^2 < T} 
 R^{\frac2{n+4}} ||g||_{L^{n+4}(B_R(x) \times (\frac{R^2}{2}, R^2))}  <\infty\}.
\end{align*}

Now we are in a position to formulate our main Theorem of this subsection.
\begin{theorem}\label{mcf}
There exists $\eps>0$, $C>0$ such that for every map $f_0:\R^n\rightarrow \R^m$ satisfying $||f_0||_{C^{0,1}(\R^n,\R^m)}<\eps$ there exists a global analytic solution $f\in X_\infty$ of \eqref{meanh} with $f(\cdot,0)=f_0$ and $||f||_{X_\infty}\le C||f_0||_{C^{0,1}(\R^n,\R^m)}$. The solution is unique in the ball $B^{X_\infty}(0,C\eps)=\{f|\ \ ||f||_{X_\infty}\le C\eps\}$.

More precisely, there exists $R>0$, $c>0$ such that for every $k\in \N_0$ and every multiindex $\alpha \in \N^n_0$ we have the estimate 
\begin{align}
 \sup_{x\in \R^n}\sup_{ t>0} |(t^{\frac{1}{2}} \nabla)^\alpha (t\partial_t)^k \nabla f(x,t)|\le c||f_0||_{C^{0,1}(\R^n,\R^m)}R^{|\alpha|+k}(|\alpha|+k)!. \label{anaMCF}
\end{align}  
Moreover the solution $f$ depends analytically on $f_0$. 
\end{theorem}
In the case $m=1$ Ecker \& Huisken \cite{ecker89,ecker91} showed the existence of a global solution of the mean curvature flow of entire graphs for any initial data which is locally Lipschitz. 
\newline
We remark that for $m>1$, one needs at least a certain "smallness" condition for the Lipschitz norm of the initial data in view of an example (due to Lawson \& Osserman \cite{lawson77}) of a minimal graph $F:\R^4 \rightarrow \R^7$ which is Lipschitz continuous but not $C^1$. 

For compact manifolds and Lipschitz initial data $f_0$ with locally small Lipschitz norm, Wang \cite{wang04a} showed the existence of a local smooth solution of the mean curvature flow.
Moreover, for $M=\Sigma_1 \times \Sigma_1$, where $\Sigma_1$ and $\Sigma_2$ are compact manifolds of constant curvature, and initial maps $f_0:\Sigma_1\rightarrow \Sigma_2$ which are Lipschitz with small Lipschitz norm, the mean curvature flow has been studied by Wang \cite{wang02} (see also \cite{tsai04}). 

In the special case $m=n$ and $f_0=\nabla u_0 \in C^{0,1}$ for some $u_0:\R^n \rightarrow \R$ (so called Lagrangian graphs) satisfying $-(1-\delta)id \le \nabla^2 u_0 \le (1-\delta)id$, where $0<\delta<1$ is arbitrary, a global smooth solution of the Lagrangian mean curvature flow for entire graphs has recently been constructed by Chau, Chen \& He \cite{chau}. 

In order to prove Theorem \ref{mcf} we start with the following Lemma.
\begin{lemma}\label{itmcf}
For every $0<T\le \infty$ and every $\gamma<1$ the operator $M[\cdot]:X^{\gamma}_T=\{f\in X_T| ||f||_{X_T}<\gamma\} \rightarrow Y_T$ is analytic and we have the estimates
\begin{align}
 ||M[f]||_{Y_T} \le c||f||^2_{X_T} \label{estyxmcf}
\end{align}
for all $f\in X^\gamma_T$ and
\begin{align}
||M[f_1]-M[f_2]||_{Y_T} \le c(\gamma)(||f_1||_{X_T}+||f_2||_{X_T})||f_1-f_2||_{X_T} \label{estyxamcf}
\end{align}
for all $f_1,f_2 \in X^\gamma_T$.
\end{lemma}
\begin{proof}
This is a consequence of the facts that for every $f\in X_T^{\gamma}$ we have
\begin{align*}
||g^{ij}-\delta^{ij}||_{L^\infty(\R^n)}\le c(||\nabla f||_{L^\infty(\R^n)})||\nabla f||_{L^\infty(\R^n)}
\end{align*}
and
\begin{align*}
||g_1^{ij}-g_2^{ij}||_{L^\infty(\R^n)}\le c(||\nabla f_1||_{L^\infty(\R^n)},||\nabla f_2||_{L^\infty(\R^n)})||\nabla (f_1-f_2)||_{L^\infty(\R^n)},
\end{align*}
where $g_l=\delta+\langle \nabla f_l,\nabla f_l \rangle$, $l\in \{1,2\}$.
\end{proof}
Next we have 
\begin{lemma}\label{mcfyx}
Let $0<T\le \infty$, $f_0\in C^{0,1}(\R^n,\R^m)$ and $M\in Y_T$. Then every solution $f$ of $(\partial_t-\Delta)f=M$ with $f(\cdot,0)=f_0$ is in $X_T$ and we have
\begin{align}
||f||_{X_T}\le c(||f_0||_{C^{0,1}(\R^n,\R^m)}+||M||_{Y_T}). \label{crucmcf}
\end{align}
\end{lemma}
\begin{proof}
First of all we observe that by Lemma \ref{homog2nd} and the above remark we can assume without loss of generality that $f_0=0$. From the translation and scaling invariance it follows that we only have to show that for some $T>1$ we have
\begin{align*}
|\nabla f(0,1)|+||\nabla^2 f||_{L^{n+4}(B_1(0)\times(\frac{1}{2},1))}\le c||M||_{Y_T}.
\end{align*}
The proof of this estimate follows from arguments similar to the ones used in the proof of Lemma \ref{xy}. Namely, we decompose
\begin{align*}
|\nabla f(0,1)|\le& |\int_Q \nabla \Phi(x,1-t)M(x,t) dx dt|+|\int_{\R^n\times(0,1)\backslash Q} \nabla \Phi(x,1-t)M(x,t) dx dt|\\
=&I+II
\end{align*}
where again $Q=B_1(0)\times [\frac12,1)$, and we estimate ($||\nabla \Phi||_{L^{\frac{n+4}{n+3}}(\R^n \times (0,1))}\le c$)
\begin{align*}
I\le& c||\nabla \Phi||_{L^{\frac{n+4}{n+3}}(Q')}||M||_{L^{n+4}(Q)}\le c||M||_{Y_T}.
\end{align*}
Moreover, we use \eqref{phi} to get
\begin{align*}
II\le c\sup_{y\in \Z^n}\sum_{m=0}^\infty \int_{2^{-m-1}}^{2^{-m}} \int_{B_1(y)}|M(x,t)|dx dt.
\end{align*} 
Next we claim that there exists $0<\gamma<1$ such that
\begin{align*}
\int_{2^{-m-1}}^{2^{-m}} \int_{B_1(0)}|M(x,t)|dx dt\le c\gamma^m ||M||_{Y_T}
\end{align*}
which then finishes the proof of the $L^\infty$-estimate. In order to proof this claim we cover $B_1(0)\times (2^{-m-1},2^{-m})$ by approximately $2^{\frac{nm}{2}}$ cylinders of the form $Q_m(y):= B_{2^{-\frac{m}{2}}}(y)\times (2^{-m-1},2^{-m})$ and we use H\"older's inequality to estimate
\begin{align*}
||M||_{L^1(Q_m(y))}\le c2^{\frac{2m-m(n+2)(n+3)}{2(n+4)}}||M||_{Y_T}\le c2^{\frac{-m(n+1)}{2}}||M||_{Y_T}
\end{align*}
and hence this proves the claim with $\gamma=\frac{\sqrt{2}}{2}$.

In order to finish the proof of the Lemma it remains to show that 
\begin{align}
||\int_0^t \int_{\R^n}\nabla^2 \Phi(x-y,t-s)M(y,s) dy ds||_{L^{n+4}(Q)}\le& c||M||_{Y_T}.\label{restmcf}
\end{align}
By using similar arguments as above we get 
\begin{align*}
\sup_{(x,t)\in Q}|\int_{\R^n\times(0,1)\backslash B_2(0)\times (\frac14,1)} \nabla^2 \Phi(x-y,t-s)M(y,s) dy ds|\le c||M||_{Y_T}
\end{align*}
and therefore we can assume that the support of $M$ is contained in $B_2(0)\times (\frac14,1)$. In this situation we can use an argument involving singular integrals as in the proof of Lemma \ref{xy} to finish the proof of \eqref{restmcf}.
\end{proof}
Theorem \ref{mcf} now follows from an application of the Banach fixed point theorem and the implicit function theorem as in the proof of Theorem \ref{main}.

Arguing as in the case of the Willmore flow we get an existence result for self-similar solutions of the mean curvature flow for entire graphs in higher codimensions. 
\begin{corollary}\label{selfmcf}
  There exists $\varepsilon>0$, $C>0$ such that if $f_0\in C^{0,1}(\R^n,\R^m)$ is
  self-similar (i.e. $f_0(x)=\frac{1}{\lambda}f_0(\lambda x)$ for every $x\in \R^n$, $\lambda>0$) with $||f_0||_{C^{0,1}(\R^n,\R^m)}<\varepsilon$, then there
  exists an analytic, self-similar solution $f\in X_\infty$ of the
  mean curvature flow \eqref{meanh} which satisfies the estimates $||f||_{X_\infty}\le C||f_0||_{C^{0,1}(\R^n,\R^m)}$ and \eqref{anaMCF}. The solution is unique in the ball $B^{X_\infty}(0,C\eps)$. Moreover $f$ can be written as $f(x,t)=\sqrt{t}\xi(\frac{x}{\sqrt{t}})$, where $\xi$ is an analytic solution of the elliptic system
\begin{align*}
h^{ij}(y)\nabla^2_{ij} \xi(y)+\frac{1}{2}(y\cdot \nabla \xi-\xi)=0,
\end{align*}
where $h_{ij}=\delta_{ij}+\langle \nabla_i \xi, \nabla_j \xi \rangle$.
\end{corollary}

\section{Harmonic map flow}
\setcounter{equation}{0}
In this section we study the harmonic map flow for maps from the euclidean space into a smooth and compact Riemannian manifold $N$, which we assume to be isometrically embedded into some euclidean space $\R^l$. For simplicity we assume first that $N$ is the round sphere $S^{l-1}\subset \R^l$ and later on we show how to extend the results to the general case. A map $u:\R^n\times [0,T) \rightarrow S^{l-1}$ is a solution of the harmonic map flow with initial condition $u_0:\R^n\rightarrow S^{l-1}$ if
\begin{align}
(\partial_t-\Delta)u&=u|\nabla u|^2\ \ \text{in}\ \ \R^n\times (0,T)\ \ \text{and}\nonumber\\
u(\cdot,0)&=u_0. \label{hflow}
\end{align}
Our main goal in this subsection is to prove a local existence result for solutions of \eqref{hflow} in the case where $u_0$ is a small $L^\infty$-perturbation of an uniformly continuous map. 

We note that the harmonic map flow is invariant under the scaling ($\lambda>0$)
\begin{align}
u_\lambda(x,t)=u(\lambda x,\lambda^2 t)\label{scalingH}
\end{align}
and we define for every $0<T\le \infty$ the function spaces 
\begin{align*}
X_T=&\{u |\,\ ||u||_X=\sup_{0<t<T}(||u(t)||_{L^\infty(\R^n)}+t^{\frac{1}{2}}||\nabla u(t)||_{L^\infty(\R^n)})\\
&+\sup_{x\in \R^n}\sup_{0<R^2<T}R^{-\frac{n}{2}} ||\nabla u||_{L^2(B_R(x)\times (0,R^2))}<\infty \}\ \ \text{and}\\
Y_T=&\{f| \,\ ||f||_{Y_T}=\sup_{0<t<T} t|| f(t)||_{L^\infty(\R^n)}\\
&+\sup_{x\in \R^n}\sup_{0<R^2<T}R^{-n} ||f||_{L^1(B_R(x)\times (0,R^2))}<\infty\}.
\end{align*}
Similar function spaces have been used in \cite{koch01} to construct a solution to the Navier-Stokes equation. 

Now we can formulate our main Theorem of this subsection.
\begin{theorem}\label{localharmonic}
There exists $\eps_0=\eps_0(n)>0$ such that for every uniformly continuous map $w:\R^n\rightarrow S^{l-1}$ and every map $u_0:\R^n\rightarrow S^{l-1}$ satisfying $||u_0-w||_{L^\infty(\R^n)}<\varepsilon_0$ there exists $\delta=\delta(\eps_0,w)>0$ and an analytic solution $u\in \varphi_\delta+X_{\delta^2}$ of \eqref{hflow}. Here $\varphi_\delta =\int_{\R^n}\Phi(\cdot-y,\delta^2)w(y) dy$. 
\end{theorem}
As a corollary of this Theorem and its proof we get 
\begin{corollary}\label{globalh}
There exists $\varepsilon_0>0$ such that for all $u_0:\R^n\rightarrow S^{l-1}$ satisfying $||u_0-P||_{L^\infty(\R^n)}<\eps_0$, where $P\in S^{l-1}$ is some arbitrary point, there exists a global analytic solution $u\in P+ X_{\infty}$ of \eqref{hflow}.
\end{corollary}
We remark that the harmonic map flow for smooth initial maps whose image lies in a geodesic ball has previously been studied by Jost \cite{jost81}.

In the following we let $w:\R^n\rightarrow S^{l-1}$ be a fixed uniformly continuous map and we let $\varphi:\R^n \times [0,\infty)\rightarrow \R^l$ be the unique solution of 
\begin{align}
(\partial_t-\Delta) \varphi&=0\ \ \ \text{in}\ \ \R^n\times (0,\infty)\ \ \ \text{and} \nonumber \\
\varphi(\cdot,0)&=w.\label{var}
\end{align} 
Since $w$ is uniformly continuous we know that for every $\eps>0$ there exists $\delta>0$ such that for every $x\in \R^n$ we have $\text{osc}_{B_\delta(x)}w \le \eps$ and therefore we get for all $x,y\in \R^n$  
\begin{align}
|w(x)-w(y)|\le \eps(1+\frac{|x-y|}{\delta}). \label{osc}
\end{align}
Now we have the following Lemma.
\begin{lemma}\label{varphi}
Let $w$ and $\varphi$ be as above. Then we have
\begin{align}
||\varphi_\delta||_{L^\infty(\R^n)}\le& c\ \ \ \text{and} \label{estvarphi1}\\ 
||\varphi_\delta-w||_{L^\infty(\R^n)}+\delta ||\nabla \varphi_\delta||_{L^\infty(\R^n)}+\delta^2||\nabla^2 \varphi_\delta||_{L^\infty(\R^n)} \le& c\eps, \label{estvarphi2}
\end{align}
where $\varphi_\delta=\varphi(\cdot,\delta^2)$.
\end{lemma}
\begin{proof}
\eqref{estvarphi1} follows from Lemma \ref{homog2nd}. For the second term in \eqref{estvarphi2} we note that for every $x\in \R^n$ we have
\begin{align*}
|\nabla \varphi_\delta(x)|=& |\int_{\R^n}\nabla \Phi(x-y,\delta^2)(w(y)-w(x))dy|\\
\le& c\eps \delta^{-n}\int_{\R^n}\frac{|x-y|}{\delta^{2}}e^{\frac{-|x-y|^2}{4\delta^2}}(1+\frac{|x-y|}{\delta})dy\\
\le& c\eps \delta^{-1},
\end{align*}
where we used \eqref{osc} in the first estimate. The first and third term in \eqref{estvarphi2} are estimated similarly.
\end{proof}

Next we assume that $u$ is a solution of \eqref{hflow} and we let $v(x,t)=u(x,t)-\varphi_\delta(x)$. From this definition it follows that $v$ is a solution of the system
\begin{align} 
(\partial_t-\Delta)v=& v|\nabla v|^2+2v\langle \nabla v, \nabla \varphi_\delta \rangle +\varphi_\delta|\nabla v|^2+v|\nabla \varphi_\delta|^2\nonumber \\
&+2\varphi_\delta\langle \nabla v,\nabla \varphi_\delta \rangle +\varphi_\delta |\nabla \varphi_\delta|^2-\Delta \varphi_\delta \nonumber \\
=:& H[v,\varphi_\delta],\nonumber\\ 
v(\cdot,0)=v_0=& u_0-\varphi_\delta. \label{phflow}
\end{align}
By \eqref{estvarphi2} we get
\begin{align*}
||v_0||_{L^\infty(\R^n)}\le ||u_0-w||_{L^\infty(\R^n)}+c\eps
\end{align*}
and hence we see that Theorem \ref{localharmonic} will be a consequence of the next Proposition if we choose $\eps$ small enough.
\begin{proposition}\label{v}
There exists $\eps_0=\eps_0(n)>0$ such that for all $v_0:\R^n\rightarrow \R^{l}$ satisfying $||v_0||_{L^\infty(\R^n)}<\varepsilon_0$ there exists $\delta=\delta(\eps_0,w)>0$ and a unique and analytic solution $v\in X_{\delta^2}$ of \eqref{phflow}.
\end{proposition}
In order to prove this Proposition we need the following two Lemmas. 
\begin{lemma}\label{firstha}
Let $\varphi_\delta =\int_{\R^n}\Phi(\cdot-y,\delta^2)w(y) dy$ and let $v\in X_{\delta^2}$. Then we have that $H[v,\varphi_\delta] \in Y_{\delta^2}$ with
\begin{align}
||H[v,\varphi_\delta]||_{Y_{\delta^2}} \le c \Big(\eps+||v||_{X_{\delta^2}}+||v||^2_{X_{\delta^2}}\Big)||v||_{X_{\delta^2}}+c\eps.\label{firstha1}
\end{align}
Moreover there exists $\eps_1>0$ and $q<1$ such that for all $\eps<\eps_1$ and all $v_1,v_2\in X_{\delta^2}^{\eps_1}=\{v\in  X_{\delta^2}|\,\ ||v||_{X_{\delta^2}}<\eps_1\}$ we have
\begin{align}
||H[v_1,\varphi_\delta]-H[v_2,\varphi_\delta]||_{Y_{\delta^2}}\le q||v_1-v_2||_{X_{\delta^2}}. \label{firstha2}
\end{align}  
\end{lemma}
\begin{proof}
This is a direct consequence of the definition of the function spaces $X_{\delta^2}$ and $Y_{\delta^2}$, the explicit expression for $H[v,\varphi_\delta]$ and Lemma \ref{varphi}. 
\end{proof}
\begin{lemma}\label{firsthaa}
Let $H\in Y_T$ for some $0<T\le \infty$. Then every solution $v$ of $(\partial_t-\Delta)v=H$ with $v(\cdot,0)=v_0\in L^\infty(\R^n,\R^l)$ is in $X_T$ and we have the estimate
\begin{align}
||v||_{X_T}\le c(||v_0||_{L^\infty(\R^n)}+||H||_{Y_T}).\label{firsthaa1}
\end{align}
\end{lemma}
\begin{proof}
Lemma \ref{homog2nd} shows that  without loss of generality we can assume that $v_0=0$. In order to finish the proof of the Lemma we argue as in \cite{koch01}. From the translation and scaling invariance of the estimate \eqref{firsthaa1} it follows that we only have to show that ($T>1$)
\begin{align*}
|v(0,1)|+|\nabla v(0,1)|+||\nabla v||_{L^2(B_1(0)\times (0,1))}\le c||H||_{Y_T}.
\end{align*}

Without loss of generality we can assume that $H$ has compact support in $\R^n\times (0,1)$. The estimate for $|v(0,1)|$ follows directly from the estimate for the heat kernel and the estimate for $|\nabla v(0,1)|$ can be shown as in \cite{koch01}. Finally, in order to get the estimate for $||\nabla v||_{L^2(B_1(0)\times (0,1))}$, we multiply the equation by $\eta^2v$, where $\eta$ is as in the proof of Lemma \ref{homog2nd}, and integrate by parts to get
\begin{align*}
\partial_t\int_{\R^n}\eta^2 |v|^2 +\int_{B_1(0)}|\nabla v|^2 \le c\int_{B_2(0)}(|v|^2+|v||H|)\le c||H||^2_{Y_T}.  
\end{align*}
Integrating over $t$ from $0$ to $1$ yields the desired result.
\end{proof}
Proposition \ref{v} (and therefore also Theorem \ref{localharmonic}) is now a consequence of the previous two Lemmas and a fixed point (respectively implicit function theorem) argument similar to the one used in the proof of Theorem \ref{main}.   

\begin{bem}
The above argument directly extends to the harmonic map flow for maps from $\R^n$ into an arbitrary compact submanifold $N$ of some euclidean space. The regularity of the solution will then also depend on the regularity of $N$ (for example the solution will be analytic if $N$ is analytic).
\end{bem}

\appendix
\numberwithin{equation}{section}
\section{The fundamental solution of the biharmonic heat equation}

The fundamental solution $b(x,t)$ of the biharmonic heat equation
\[ 
u_t + \Delta^2 u = 0 
\]
can be expressed through the Fourier integral 
\[ 
g(x) =  (2\pi)^{-\frac{n}{2}}\int_{\R^n} e^{ikx -|k|^4} dk 
\]
by defining $b(x,t) = t^{-\frac{n}{4}}  g(xt^{-\frac{1}{4}})$. The function $g$ is smooth and radial. 

In the following we want to apply the method of the stationary phase to study the behavior of $g(x)$ as $|x|\rightarrow \infty$. The asymptotics of $g$ are determined by the complex critical points of the complex phase function $p(k)=ik - |k|^4$ which are given by 
\[ 
k_\pm  = (\pm \frac{\sqrt{3}}2 + \frac{1}{2}i) 2^{-\frac{2}{3}} |x|^{\frac{1}{3}}, k_0= -i 2^{-\frac{2}{3}} |x|^{\frac{1}{3}}.  
\] 
The values of the function $q(k)=ik x-k^4$ at the critical points $k_\pm$ of $p$ are 
\[ 
q(k_\pm)=ik_\pm x  - k_{\pm}^4 =  - \frac34  i x k_\pm  =  
- 2^{\frac{1}{3}} \left(  \frac3{16} \pm \frac{3^{\frac{3}{2}}}{16}i \right) |x|^{\frac43}. 
\]
Moreover the Hessian of the phase function is given by 
\[ 
\nabla^2_{ij}p(k)=- 4 k^2 \delta_{ij} - 6 (k_i k_j) .
\] 
To simplify the notation we will restrict ourselves to the case $x = (r,0)$. Next we calculate the eigenvalues of the Hessian $\nabla^2 p$ at the critical points $k_\pm$ to be 
\[
-2 \left(\frac12 \pm \frac{\sqrt{3}}{2}i \right) 2^{-\frac{1}{3}}
|r|^{\frac{2}{3}} 
\]
and 
\[
-3 \left(\frac12 \pm \frac{\sqrt{3}}{2}i \right)
2^{-\frac{1}{3}} |r|^{\frac{2}{3}}, 
\] 
where the second one has multiplicity $n-1$. 
Hence the oscillatory integral $g$ is given as the real part of a complex function $\mathbf g$ which satisfies
\begin{align} 
 \left( \left(\frac{1}{2} + \frac{\sqrt{3}}{2}i \right)^{-\frac12}\right)^n  |x|^{\frac{n}{3}} \exp\left( 2^{\frac{1}{3}}\left(  
    \frac3{16} +\frac{3^{\frac{3}{2}}}{16}i  \right) |x|^{\frac43}\right)  \mathbf g
  \sim \pi^{\frac{n}{2}} + O(|x|^{-1}).\label{asymp}
\end{align}     
This is an asymptotic relation which remains true after differentiating both sides.
 
A rigorous proof of this asymptotic formula can be given as follows. We recall that $x=(r,0)$ with $r >>1$ and we shift the domain of integration to $\R^n 
+ i2^{-\frac{5}{3}}   r^{\frac{1}{3}}e_1 $. We obtain 
\begin{align*}
 g(x) =& (2\pi)^{-\frac{n}{2}} \exp\left( -2^{\frac{1}{3}}   \frac3{16} |x|^{\frac43}\right)                
\int_{\mathbb{R}^n}\Big[  \exp \Big(  i (x\xi -  2^{-\frac{1}{3}}  |\xi|^2 \xi_1 r + 8^{-1} \xi_1 r^3)  \\
&- (|\xi|^2- 3 2^{-\frac{10}{3}} |x|^{\frac{2}{3}})^2 - 2^{-\frac{4}{3}} \xi_2^2 |x|^{\frac{2}{3}} \Big) \Big] 
d\xi. 
\end{align*} 
The asymptotic relation \eqref{asymp} is now obtained by a standard evaluation of the oscillatory integral as in Theorem 7.7.5 of \cite{hoer}.

\end{document}